\documentclass[a4paper,10pt]{article}
\usepackage{amsfonts}
\usepackage{microtype}
\usepackage{makeidx}
\usepackage{setspace}
\usepackage{graphicx}
\usepackage[all]{xy}
\usepackage{amsmath}
\usepackage{amssymb}
\usepackage{booktabs}
\usepackage{braket}
\usepackage{array}
\usepackage{tabularx}
\usepackage{multirow}
\usepackage{indentfirst}
\usepackage{cite}
\usepackage{stmaryrd}
\usepackage{faktor}
\usepackage{titling}
\usepackage{tikz}
\usepackage{dsfont}
\usetikzlibrary{matrix,arrows,decorations.pathmorphing}
    \renewcommand{\leq}{\leqslant}
    \renewcommand{\geq}{\geqslant}
\usepackage{amsthm}
\theoremstyle{plain}
\newtheorem{thm}{Theorem}[section]
\newtheorem{dfn}[thm]{Definition}

\newtheorem{prop}[thm]{Proposition}
\newtheorem{cor}[thm]{Corollary}
\newtheorem{ex}[thm]{Example}
\newtheorem{conge}[thm]{Conjecture}

\theoremstyle{plain}
\newtheorem{oss}[thm]{Remark}

\DeclareMathOperator{\spn}{span} 

\DeclareMathOperator{\inv}{\mathrm{inv}}

\DeclareMathOperator{\pfl}{\mathrm{Fl_P}}
\DeclareMathOperator{\gl}{\mathrm{GL}}

\DeclareMathOperator{\gr}{\mathrm{Gr}}
\DeclareMathOperator{\cs}{\mathrm{Conf}}

\DeclareMathOperator{\inve}{\mathrm{inv}}
\DeclareMathOperator{\pos}{\mathrm{POS}}

\DeclareMathOperator{\id}{\mathrm{Id}}

\DeclareMathOperator{\aut}{\mathrm{Aut}}

\DeclareMathOperator{\F}{\mathbb{F}}
\DeclareMathOperator{\R}{\mathbb{R}}
\DeclareMathOperator{\C}{\mathbb{C}}
\DeclareMathOperator{\N}{\mathbb{N}}

\DeclareMathOperator{\JP}{\mathcal{J}(P)}
\DeclareMathOperator{\JQ}{\mathcal{J}(Q)}
\DeclareMathOperator{\End}{\mathrm{End}}
\DeclareMathOperator{\soc}{\mathrm{soc}}

\begin{document}

\begin{center}

{\Large \bf P-flag spaces and incidence stratifications}

\vspace{1cm}

 Davide Bolognini\footnote{Dipartimento di Matematica,
Universit\`a di Bologna, Bologna, Italy,
{\em davide.bolognini.cast@gmail.com}}
$\,$ Paolo Sentinelli\footnote{Dipartimento di Matematica,
Politecnico di Milano, Milan, Italy, {\em paolosentinelli@gmail.com}},

\end{center}

 \vspace{0.5cm}

\begin{abstract}
For any finite poset $P$, we introduce a homogeneous space as a quotient of the general linear group. When $P$ is a chain this quotient
is a complete flag variety. Moreover, we provide partitions for any set in a projective space, induced by the action of incidence groups of
posets. Our general framework allows to deal with several combinatorial and geometric objects, unifying and extending different structures such as Bruhat orders, parking functions and weak orders on matroids. We introduce the notion of $P$-flag matroid, extending flag matroids.
\end{abstract}

 \vspace{0.5cm}

\section{Introduction}

Flag varieties are classical homogeneous spaces, studied from several different points of view. They parameterize the flags of $V$, i.e.\ sequences of subspaces $V_1 \subseteq \ldots \subseteq V_n=V$, where $V$ is an $n$-dimensional $\F$-vector space and $V_i \subseteq V$ is an $i$-dimensional subspace of $V$. They can be obtained as quotients of the general linear group $GL(n,\F)$ with the subgroup $B$ of invertible upper triangular matrices. The action of $B$ on the flags gives rise to a partition into {\em Schubert cells}. Their Zariski closures are the so-called {\em Schubert varieties}, which are in bijection with the symmetric group $S_n$. The {\em Bruhat order} on $S_n$ is the poset of Schubert varieties ordered by inclusion. Similar facts hold for the Grassmannian $\gr_{\F}(k,n)$, replacing $S_n$ with the subset $S_n^{(k)}$ of Grassmannian permutations.

In this article we introduce a new class of homogeneous spaces, namely the quotients $\pfl(\F):=\gl(n,\F)/I^*(P;\F)$, where $I^*(P;\F)$ is the so-called \emph{incidence group}
of a poset $P=\left(\{1,\ldots,n\},\leqslant_P\right)$, i.e. the group of invertible elements of the incidence algebra of $P$.
Since the Borel subgroup $B$ is the group $I^*(c_n;\F)$, where $c_n$ is the chain on $n$ elements, we recover classical flag varieties.

In Definition \ref{complete flag} we introduce $P$-flags in $V$. They are tuples $(V_1,\cdots,V_n)$ of vector subspaces of $V$, satisfying, among others, the following properties (see Proposition \ref{proprieta bandiere}):
\begin{itemize}
    \item $V_i \subseteq V_j$ if and only if $i \leqslant_P j$;
    \item $\mathrm{dim}(V_i)=|i^\downarrow|$, where $i^\downarrow=\left\{j\in [n]: j\leqslant_P i\right\}$.
\end{itemize}
We prove that $\pfl(\F)$ is a homogeneous space parametrizing $P$-flags in $V$ (Theorem \ref{fl omogeneo}).
For this reason, we call $\pfl(\F)$ the $P$-{\emph{flag space} over $\F$.
The elements of these spaces are certain spanning subspace configurations, see Remark \ref{oss spanning configurations}.

The second main contribution of this paper is a new tool to obtain finite partitions of any subset $X$ of a projective space, introducing the notion of {\em incidence stratifications}, see Definition \ref{def incidence strat}. In fact, an incidence group $I^*(Q;\F)$, where $Q$ is a poset of cardinality $n$, acts on the projective space $\mathbb{P}(\F^n)$ by left multiplication; the orbits of this action  are indexed by non-empty order ideals of $Q$ (Theorem \ref{pspazio}).

This general framework allows to deal with several combinatorial and geometric objects, unifying and extending different structures such as Bruhat orders, parking functions and weak orders on matroids.

For the Grassmannian $\gr_{\F}(k,n) \hookrightarrow \mathbb{P}\left(\bigwedge^kV\right)$, we consider a suitable poset $Q^k_<$ (Definition \ref{ordine k}) to realize an incidence stratification. In this setting, for $Q=c_n$, we recover the classical Schubert cells (Proposition \ref{celle classiche GR}). When $Q=t_n$, the incidence group $I^*((t_n)^k_<;\F)$ is a maximal torus and we obtain the matroid strata introduced in \cite{Gelfand} and studied, e.g. in  \cite{Torus GP}, \cite{Ford}, \cite{Sturmfels}, \cite{transcendence degree}. See also \cite[Section~2.4]{orientedMat} and references there.

We introduce a poset $Q^P$ (Definition \ref{def Q alla P}) to provide an incidence stratification of the $P$-flag space $\pfl(\F) \hookrightarrow \mathbb{P}\left(\bigotimes\limits_{i=1}^n\bigwedge^{|i^{\downarrow}_P|}V \right)$. In this way the Schubert stratification of a flag variety is recovered, see Proposition \ref{celle classiche F}.

One more contribution is the construction, for $\pfl(\F)$, of the $Q$-\emph{Bruhat poset}, whose elements are order ideals of $Q^P$ (Definition \ref{QP Bruhat order}).
In a completely combinatorial way, we obtain the Bruhat order of $S_n$ as the $c_n$-Bruhat order on the classical flag variety, see Proposition \ref{corollario iniezione2}.
The study of the $t_n$-Bruhat poset of $\pfl(\F)$ is one motivation to introduce $P$-\emph{flag matroids} (Definition \ref{definizione P flag matroid}), extending flag matroids.
The representable ones (Definition \ref{def P-flag matroid rep}) determine the $t_n$-stratification of $\pfl(\F)$ (Corollary \ref{cor P-flag matroidi}).
In general, we believe that a $Q$-Bruhat poset is graded, see Conjecture \ref{congettura 2}.

The paper is organized as follows: 

\begin{itemize}
\item In Section \ref{sezione preliminari} we fix notation and we recall useful facts concerning symmetric groups, incidence algebras and matroids. Several classical topics
overviewed in the section are extended in this paper.

\item In Section \ref{sezione P-flag} we introduce $P$-flags in a vector space (Definition \ref{complete flag}). We prove that $P$-flags are parameterized by a homogeneous space $\pfl(\F)$ (Theorem \ref{fl omogeneo}), which for $\F=\mathbb{R}$ is a differentiable manifold (Corollary \ref{dimensione}).

    In Theorem \ref{canonical}, we describe as homogeneous spaces some of the orbits of the action of $I^*(Q,\F)$ on $P$-flags, where $Q$ is any poset of cardinality $n$.
     This shows that also classical Schubert cells are homogeneous spaces, where the isotropy subgroups are the incidence groups of posets whose Hasse diagrams are the graphs introduced in \cite{Forest-like permutations}, see Remark \ref{oss forest-like}.


\item Section \ref{sezione grassmanniane} and Section \ref{sezione strati pflag} are devoted to the study of incidence stratifications of Grassmannians $\gr_{\F}(k,n)$ and $P$-flag spaces. First we provide full information about the orbits (and their Zariski closures) of the action of $I^*(Q,\F)$ on $\mathbb{P}(V)$ (Theorem \ref{pspazio}). Then we characterize $Q$-Schubert cells in both cases, indexing them with order ideals in suitable posets; the characterization is given in terms of representable matroids (Theorem \ref{teorema matroidi}) and sets represented by $P$-flags (Definition \ref{def flag repr} and Theorem \ref{teorema P-matroidi bandiera}). We introduce the $Q$-Bruhat posets of $\gr_{\F}(k,n)$  and $\pfl(\F)$ (Definitions \ref{Q Bruhat} and \ref{QP Bruhat order}). The $c_n$-Bruhat orders coincide with the Bruhat order on $S_n^{(k)}$ and $S_n$, respectively (Propositions
    \ref{corollario iniezione} and \ref{corollario iniezione2}).
     On the other hand, the $t_n$-Bruhat order of $\gr_{\F}(k,n)$ is the so-called weak order on representable matroids of rank $k$, see Remark \ref{oss weak order matroidi}.

The last part of the paper describes the $c_n$-stratification of $\mathrm{Fl}_{t_n}(\F)$ in terms of (dual) parking functions
(Theorem \ref{proposizione parking}).
\end{itemize}

\section{Notation and preliminaries} \label{sezione preliminari}
In this section we fix notation and recall some definitions useful for the rest of the paper. We refer to \cite{Incidence algebras} and \cite{StaEC1} for posets and their incidence algebras, to \cite{BB} and  \cite{humphreysCoxeter} for the theory of Coxeter groups, to \cite{orientedMat}, \cite{coxeter matroids} and \cite{oxley} for matroids and flag matroids, to \cite{brion}, \cite{LakshmibaiGrassmannian}, \cite{Interplay},  and \cite{indiano} for general results on Grassmannians and flag varieties.

Let $\mathbb{N}$ be the set of non-negative integers. For $n\in \mathbb{N}\setminus \left\{0\right\}$, we use the notation $[n]:=\left\{1,2,\ldots,n\right\}$. For a finite set $X\neq \varnothing$, we denote by $|X|$ its cardinality, by $\mathcal{P}(X)$ its power set, by $X^n$ its $n$-th power under Cartesian product and we let $X^0:=\left\{()\right\}$.
If $x\in X^n$, we denote by $x_i$ the projection of $x$ on the $i$-th factor.
The $q$-analog of $n$ is a polynomial defined by $[n]_q:=\sum\limits_{i=0}^{n-1}q^i$;
the $q$-analog of the factorial is the polynomial $[n]_q!:=\prod\limits_{k\in [n]}[k]_q$.
Let $k\in \mathbb{N}$ with $k\leqslant n$. We define the set $$[n]_<^k:=\left\{(x_1,\ldots,x_k) \in [n]^k:x_1<x_2<\ldots<x_k \right\}.$$
It is clear that there exists a bijection $\bigcup\limits_{k=0}^n[n]_<^k \rightarrow \mathcal{P}([n])$. Hence, the Boolean operations on $\mathcal{P}([n])$ make sense in $\bigcup\limits_{k=0}^n[n]_<^k$.

The notations $\End(O)$  and $\aut(O)$ stand for the set of endomorphisms and automorphisms of an object $O$ in a category.

The {\em symmetric group} of permutations of $n$ objects is denoted by $S_n$. A permutation $\sigma \in S_n$ can be written in one line notation as $\sigma(1)\sigma(2)\ldots\sigma(n)$. An \emph{inversion} in $\sigma$ is a pair $(i,j) \in [n]^2_<$ such that $\sigma(i)>\sigma(j)$. The number of inversions in $\sigma$ is denoted by $\inv(\sigma)$.

For any field $\mathbb{F}$ let $\mathrm{Mat}(n,\F)$ be the set of $n \times n$ matrices over $\F$, $\id_n$ the identity matrix and $\gl(n,\F)$ the group of invertible matrices of size $n$.

The \emph{projective space} of a vector space $V$ is denoted by $\mathbb{P}(V)$ and
the \emph{Grassmannians} by
$$\gr_{\F}(k,n):=\left\{W \subseteq \F^n:\mbox{$W$ is a vector subspace of dimension $k$}\right\}.$$
Let $\phi: \gr_{\F}(k,n) \rightarrow \mathbb{P}\bigl(\bigwedge^k \F^n \bigr)$ be the Pl\"ucker embedding, i.e.\ the  injective function defined by $\phi(W)=[w_1 \wedge \ldots \wedge w_k]$, for
any basis $\left\{w_1,\ldots,w_k\right\}$ of $W\in \gr_{\F}(k,n)$.

Finally, the set of \emph{complete flags} in $\F^n$ is
$$\mathrm{Fl}_n(\F):=\left\{W_1 \subseteq \ldots \subseteq W_n : W_i \in \gr_{\F}(i,n), \, \forall \,i\in[n]\right\}.$$

\subsection{Posets, incidence algebras and incidence groups}

All posets considered in this paper are finite. An \emph{interval} in a poset $\left(X,\leqslant\right)$ is a subset $[x,y]:=\left\{z\in X:x \leqslant z \leqslant y\right\}$, where $x,y \in X$ and $x \leqslant y$. When $|[x,y]|=2$, we use the notation $x \vartriangleleft y$. The following two posets appear repeatedly in the sequel:

\begin{itemize}
\item $c_n:=([n],\leqslant)$, the chain of $n$ elements;
\item $t_n$ the trivial poset on $[n]$, i.e.\ the poset without relations.
\end{itemize}

We need to introduce the following definition in order to deal with incidence algebras as matrix algebras.

\begin{dfn}
Let $n >0$. Define the set of \emph{naturally labeled posets} as $$\pos(n):=\left\{([n],\leqslant_P): i \leqslant_P j \Rightarrow i \leqslant j, \, \forall \, i,j \in [n]\right\}.$$ 
\end{dfn}


The set of relations of $P \in \pos(n)$ is $$T_P:=\left\{(i,j)\in [n]^2_<:i<_Pj\right\}.$$
The elements of $\pos(n)$ can be ordered by setting $$P \leqslant Q \Leftrightarrow T_P \subseteq T_Q,$$ for all $P,Q \in \pos(n)$.
This is a particular case of weak order on binary relations, as defined in \cite{integer posets}.
The poset $\left(\pos(n),\leqslant\right)$ has minimum and maximum, namely $t_n$ and $c_n$, respectively.


The following notions are fundamental for the rest of this article.

\begin{dfn}
The {\em incidence algebra} of a poset $P \in \pos(n)$ over a field $\mathbb{F}$ is $$I(P;\mathbb{F}):=\left\{A \in \mathrm{Mat}(n,\F): A_{i,j}=0, \text{   if   } i>j \text{ or } (i,j) \in [n]^2_< \setminus T_P\right\},$$ where $A_{i,j}$ is the $ij$-entry of the matrix $A$.
The \emph{incidence group} of $P$ over $\mathbb{F}$ is $$I^*(P;\mathbb{F}):=I(P;\F)\cap \gl(n,\F).$$ The \emph{unipotent group} of $P$ is the subgroup of $I^*(P;\mathbb{F})$ defined by $$U(P;\mathbb{F}):=\left\{A\in I^*(P;\mathbb{F}):A_{i,i}=1,~\forall~i\in[n]\right\}.$$
\end{dfn}
The algebra $I(t_n;\mathbb{F})$ is the algebra of diagonal matrices over $\mathbb{F}$.
In general, it is clear that  $I(P;\mathbb{F})$ is a subalgebra of the algebra $I(c_n;\mathbb{F})$ of $n \times n$ upper triangular matrices over $\mathbb{F}$.

Notice that $P \leqslant Q$ implies $I^*(P;\mathbb{F})\subseteq I^*(Q;\mathbb{F})$, for all $P,Q \in \pos(n)$.
We are going to prove that the quotient $I^*(Q;\mathbb{F})/I^*(P;\mathbb{F})$ has a nice structure, under suitable assumptions.

A graph on $n$ vertices is a pair $([n],E)$, where $E\subseteq [n]^2_<$ is the set of edges. The \emph{comparability graph} of $P\in \pos(n)$ is the graph $([n],T_P)$.

\begin{dfn}
  Let $P,Q\in \pos(n)$ such that $P\leqslant Q$. We say that $P$ is \emph{complemented} in $Q$ if $([n],T_Q\setminus T_P)$ is the comparability graph of a poset $P^c(Q)$.
\end{dfn}

\begin{prop}\label{Prop complementato}
 Let $P,Q\in \pos(n)$. Assume $P$ complemented in $Q$. Then $U(P^c(Q);\mathbb{F}) \subseteq I^*(Q;\mathbb{F})$ and we have that the canonical projection $I^*(Q;\mathbb{F}) \rightarrow I^*(Q;\mathbb{F})/I^*(P;\mathbb{F})$ restricts to a bijection $$ \pi_U : U(P^c(Q);\mathbb{F}) \rightarrow I^*(Q;\mathbb{F})/I^*(P;\mathbb{F}).$$
\end{prop}
\begin{proof}
If $P=Q$ then $P^c(Q)=t_n$ and the result follows. Assume $P<Q$.
It is clear that $P^c(Q)\leqslant Q$. Then $U(P^c(Q),\mathbb{F}) \subseteq I^*(Q,\mathbb{F})$.
Since $U(P^c(Q),\mathbb{F}) \cap I^*(P,\mathbb{F}) = \left\{\id_n\right\}$, the function $\pi_U$ is injective.


It remains to prove that $\pi_U$ is surjective. Let $A \in I^*(Q;\mathbb{F})$. It is sufficient to prove that there exists $X \in I^*(P;\mathbb{F})$ with $AX \in U(P^c(Q),\mathbb{F})$, because $\pi_U(AX)=\pi_U(A)$. The condition $AX \in U(P^c(Q),\mathbb{F})$ is satisfied if and only if
\begin{enumerate}
  \item $X_{i,i}=\frac{1}{A_{i,i}}$, for all $i\in [n]$ and
  \item $\sum \limits_{i \leqslant_Q k \leqslant_P j} A_{i,k}X_{k,j}=0$, for all $(i,j)\in T_P$.
\end{enumerate}
This gives a non homogeneous linear system whose matrix is an element of $I^*(T_P;\F)$, where $T_P$ is the induced subposet of the Cartesian product $Q \times Q$. Then such a linear system admits a solution $X$.
\end{proof}

  \begin{oss}
  The bijection of Proposition \ref{Prop complementato} is not a group isomorphism, because in general $I^*(P;\mathbb{F})$ is not a normal subgroup of $I^*(Q;\mathbb{F})$.
  \end{oss}

We end this section by defining the following duality function.
\begin{dfn} \label{dual}
  An involution $* : \pos(n) \rightarrow \pos(n)$ is defined by letting $i \leqslant_{P^*} j$ if and only if $n+1-j \leqslant_P n+1-i$, for every $P\in \pos(n)$. A fixed point of $*$ is called a \emph{self-dual poset}.
\end{dfn}

\subsection{The symmetric groups as Coxeter groups}
A {\em Coxeter system} $(W,S)$ is a group $W$ with a presentation whose generators are the elements of a finite set $S=\left\{s_1,\cdots,s_{n-1}\right\}$, with relations given by $s_i^2=e$ and $(s_is_j)^{m_{ij}}=e$, for suitable $m_{ij} \geq 2$ if $i \neq j$, where $e$ is the identity in $W$.

Given a Coxeter system $(W,S)$, the {\em length function} $\ell:W \rightarrow \N$ is defined by $\ell(w):=\min \left\{k \in \N: w=s_{i_1}s_{i_2}\cdots s_{i_k} \right\}$, for every $w \in W$.

For any $J\subseteq S$, the subgroup generated by $J$ is denoted by $W_J$. Define $$W^J:=\left\{w\in W:\ell(ws)>\ell(w),\,\forall\, s\in J\right\}.$$ We recall an important result (see \cite[Proposition~2.4.4]{BB}).

\begin{prop}
Any element $w\in W$ factorizes uniquely as $w=w^Jw_J$, where $w^J\in W^J$, $w_J\in W_J$ and $\ell(w)=\ell(w^J)+\ell(w_J)$.
\end{prop}
Therefore one can define an idempotent function $P^J: W \rightarrow W$ by setting $P^J(w):=w^J$.

One of the most important features of a Coxeter group is a natural partial order $\leqslant$ on it, called {\em Bruhat order}. It can be defined by the subword property (see \cite[Chapter~2]{BB} and \cite[Chapter~5]{humphreysCoxeter}). The induced subposet $(W^J,\leqslant)$ is graded with rank function $\ell$ (see \cite[Theorem~2.5.5]{BB}) 
and the function $P^J$ is order preserving, i.e. $u \leqslant v$ implies $P^J(u) \leqslant P^J(v)$, for all $u,v\in W$ (see \cite[Proposition~2.5.1]{BB}).

The symmetric group $S_n$ is a Coxeter group; its standard Coxeter presentation has generators $S=\left\{s_1,\ldots,s_{n-1}\right\}$,
where $s_i$ is the permutation $12\ldots(i+1)i\ldots n$, for all $i\in [n-1]$. With respect to this presentation, $\ell(\sigma)=\inv(\sigma)$, for every $\sigma \in S_n$. Hence the element of maximal length is $w_0=n(n-1)\ldots21$.

The following example should make clear how to obtain the permutation $P^J(\sigma)$. For more information about how $P^J$ rearranges a permutation, we refer to \cite[Section~2.4]{BB}.
\begin{ex}
  Let $n=7$, $J=\left\{s_1,s_2,s_4,s_6\right\}$ and $\sigma=4317625$. Therefore we have to rearrange increasingly the blocks $431$, $76$ and $25$. It follows that $P^J(\sigma)=1346725$.
\end{ex}

We denote $S_n^{S\setminus \left\{s_k\right\}}$ by $S_n^{(k)}$, for all $k\in [n-1]$, and we set $S_n^{(n)}:=\{e\}$. It is clear that $$S_n^{(k)}=\left\{\sigma\in S_n: \sigma(1)<\cdots<\sigma(k), \sigma(k+1)<\cdots<\sigma(n)\right\},$$ for all $k\in[n]$. The elements of $S_n^{(k)}$ are called \emph{Grassmannian permutations} since they index the set of Schubert varieties of $\gr_{\C}(k,n)$. Moreover the set $S_n^{(k)}$ is in bijection with the set $[n]_<^k$, for all $k\in [n]$.
For example,
$2357146 \in S_7^{(4)}$
corresponds to $(2,3,5,7)\in [7]^4_<$.

By the next result the Bruhat order on $S^{(k)}_n$ corresponds to the componentwise ordering of $[n]^k_<$ (see \cite[Proposition~2.4.8]{BB}).
\begin{prop} \label{proposizione Bruhat quoziente}
The induced Bruhat order on $S_n^{(k)}$ is described by $$\sigma \leq \tau \text{   if and only if    } \sigma(i) \leq \tau(i), \text {   for every   } i \in [k-1],$$
for all $\sigma, \tau \in S_n^{(k)}$.
\end{prop}
By \cite[Theorem~2.6.1]{BB}, the Bruhat order on $S_n$ can be given in terms of the posets $(S_n^{(k)},\leqslant)$, $k\in [n]$. Namely $\sigma \leqslant \tau$ if and only if $P^{S\setminus \left\{s_k\right\}}(\sigma) \leqslant P^{S\setminus \left\{s_k\right\}}(\tau)$, for all $k\in [n-1]$.

The poset $(S_n^{(k)},\leqslant)$ is isomorphic to the set of Schubert varieties in $\gr_{\C}(k,n)$ ordered by inclusion.
Analogously, $(S_n,\leqslant)$ is the poset of Schubert varieties in $\mathrm{Fl}_n(\C)$ ordered by inclusion.

We are also interested in the so-called {\em Gale ordering} on $S_n^{(k)}$.

\begin{dfn}\label{gale}
The \emph{Gale ordering} $\leqslant^{\sigma}$ on $S_n^{(k)}$ induced by $\sigma \in S_n$, is defined by letting
$u \leqslant^{\sigma} v$ if and only if $P^{S\setminus \left\{s_k\right\}}(\sigma u) \leqslant P^{S\setminus \left\{s_k\right\}}(\sigma v)$, for all $u,v \in S_n^{(k)}$.
\end{dfn} For example, let $u=2413567$, $v=5712346$ in $S_7^{(2)}$, and $\sigma=3256174\in S_7$. Then $u \leqslant^e v$. Moreover
$\sigma u =2635174$, $\sigma v= 1432567$ and $P^{S\setminus \left\{s_2\right\}}(\sigma v) =1423567 \leqslant 2613457=P^{S\setminus \left\{s_2\right\}}(\sigma u)$.
Therefore $v \leqslant^\sigma u$.

Following \cite[Section~1.7]{coxeter matroids}, we define the Gale order on a symmetric group.
\begin{dfn}
  The \emph{Gale ordering} $\leqslant^{\sigma}$ on $S_n$ induced by $\sigma \in S_n$, is defined by letting
$u \leqslant^{\sigma} v$ if and only if $\sigma u \leqslant \sigma v$, for all $u,v \in S_n$.
\end{dfn}
\noindent This is equivalent to require $P^{S\setminus \left\{s_k\right\}} (u) \leqslant^{\sigma} P^{S\setminus \left\{s_k\right\}}(v)$, for every $k\in [n-1]$.

For example, let $u=324561$, $v=623541$ in $S_6$, and $\sigma=325614\in S_6$. Then $u \leqslant^e v$. Moreover
$\sigma u =526143$, $\sigma v= 425163$ and $526143 \nleqslant 425163$. Hence $u \nleqslant ^\sigma v$.

\subsection{Matroids} \label{subsection matroids}

Let $n>0$ and $k\in [n]$. A set $M \subseteq S^{(k)}_n$ is a \emph{matroid}\footnote{More precisely, the set of \emph{bases} of a matroid.} of rank $k$ if it satisfies the \emph{Maximality Property}:

\begin{center} \label{maximality property}
  the induced subposet $(M,\leqslant^{\sigma})$ has maximum, for all $\sigma \in S_n$.
\end{center}

\begin{oss} \label{osserv massimi minimi}
  Since the left multiplication by $w_0$ is an antinvolution of the poset $(S_n,\leqslant)$, the Maximality Property is equivalent to saying
  that $(M,\leqslant^{\sigma})$ has minimum, for all $\sigma \in S_n$, i.e.\ has maximum and minimum, for all $\sigma \in S_n$.
\end{oss}
The set of matroids in $[n]^k_<$ can be ordered by inclusion (this is usually called \emph{weak order}, see e.g. \cite[Chapter~9]{Theory of matroids}).

Let $W \in \gr_{\F}(k,n)$ and $\left\{v_1,\ldots,v_k\right\} \subseteq \F^n$ be a basis of $W$. If $\left\{e_1,\ldots,e_n\right\}$ is the canonical basis of $\F^n$, one has that $$v_1\wedge \ldots\wedge v_k= \sum \limits_{i\in [n]^k_<} a_ie_{i_1}\wedge \ldots \wedge e_{i_k}.$$ It is well known that $M(W):=\left\{i\in [n]^k_<:a_i\neq 0\right\}$ is the set of bases of a matroid. Recall that we identify the set $S^{(k)}_n$ with $[n]_<^k$.
We say that a matroid $M \subseteq [n]^k_<$ is \emph{representable} over a field $\F$ if there exists a vector space $W \in \gr_{\F}(k,n)$ such that
$M=M(W)$. The equivalence relation $W_1 \sim W_2$ if and only if $M(W_1)=M(W_2)$, for all $W_1,W_2\in \gr_{\F}(k,n)$, provides the matroid stratification of $\gr_{\F}(k,n)$ introduced and studied in \cite{Gelfand}.

\begin{oss}
  Notice that the equivalence classes of the relation $\sim$ are given by the intersection between $\phi(\gr_{\F}(k,n))$ and the orbits of the action of the group
  of invertible diagonal matrices, of size $\binom{n}{k}$, on $\mathbb{P}\left(\bigwedge^k \F^n\right)$, where $\phi$ is the Pl\"ucker embedding.
\end{oss}

By using \cite[Theorem~3.3]{Lattice path matroids}, it is not difficult to characterize Bruhat intervals in $S_n^{(k)}$ as particular types of transversal matroids, namely lattice path matroids, in the meaning of \cite[Definition~3.1]{Lattice path matroids} (see also \cite[Definition~22]{Oh}). By \cite[Lemma~23]{Oh}
they are positroids.

We recall the following extension of the notion of matroid.
\begin{dfn} \label{def flag matr}
A subset $F \subseteq S_n$ such that the induced subposet $(F,\leqslant^\sigma)$ has maximum for all $\sigma \in S_n$, is said to be  a \emph{flag matroid}.
\end{dfn}  \noindent By \cite[Theorem~4.4]{CasDadMar}, any Bruhat interval in $S_n$ is a flag matroid. For sake of completeness we provide a proof of the following property which gives a connection between flag matroids and matroids.

\begin{prop}
Let $k \in [n-1]$. If $F \subseteq S_n$ is a flag matroid, then $\left\{P^{S\setminus \left\{s_k\right\}}(f):f \in F\right\}$ is a matroid of rank $k$.
\end{prop}

\begin{proof} We set $J_k:=S\setminus \left\{s_k\right\}$.
Let $\sigma \in S_n$ and $f_{\sigma}$ be the maximum of the poset $(F,\leqslant^\sigma)$. We claim that $P^{J_{k}}(f_{\sigma})$ is the maximum of $\left\{P^{J_{k}}(f):f \in F\right\} \subseteq S_n^{(k)}$ with respect to $\leqslant^{\sigma}$. Let $u \in F$; then $u \leqslant^{\sigma} f_{\sigma}$, i.e.\ $\sigma u\leqslant \sigma f_{\sigma}$.
Recall that the projection $P^{J_{k}}$ is order preserving. Then $P^{J_{k}}(\sigma u)\leqslant P^{J_{k}}(\sigma f_{\sigma})$. We have that $P^{J_{k}}(\sigma u)=P^{J_{k}}(\sigma u^{J_k} u_{J_k})=P^{J_{k}}(\sigma u^{J_{k}})$ and similarly $P^{J_{k}}(\sigma f_{\sigma})=P^{J_{k}}(\sigma f_{\sigma}^{J_{k}})$; then
$P^{J_{k}}(\sigma P^{J_k}(u)) \leqslant P^{J_k}(\sigma P^{J_k}(f_\sigma))$, i.e. $P^{J_k}(u) \leqslant^\sigma P^{J_k}(f_\sigma)$. This concludes the proof.
\end{proof}


\section{P-flag spaces} \label{sezione P-flag}

In this section we introduce a class of homogeneous spaces which is one the main object of our study, recovering as particular cases the flag varieties and the
moduli space of $n$ independent lines in $\C^n$.

Let $\F$ be a field, $n>0$ and $P\in \pos(n)$. Consider $V:=\F^n$, the $\F$-vector space with canonical basis $\left\{e_i:i\in [n]\right\}$.
Given any subset $I\subseteq [n]$, we define the vector subspace $$V_I:=\spn_{\F}\left\{e_i:i\in I\right\}.$$

%

Recall that an {\em order ideal} in a poset $P$ is a subset $I \subseteq P$ such that $i \in I$ and $j \leq_P i$ imply $j \in I$. The distributive lattice of order ideals of a poset $P\in \pos(n)$ is denoted by $\JP$. It is clear that there is a bijection between $\JP$ and the {\em antichains} of $P$, i.e.\ the set $\left\{\max(I):I\in \JP\right\}$.

For $i\in [n]$, we define the {\em principal order ideal} generated by $i\in P$ by setting $$i^\downarrow:=\left\{j\in [n]: j\leqslant_P i\right\}.$$
Given a subset $I \subseteq [n]$, we define $I^\downarrow :=\bigcup\limits_{i\in I}i^\downarrow$,
 the order ideal of $P$ generated by $I$. We write $i^\downarrow_P$ and $I_P^\downarrow$ whenever we need to stress the poset under consideration.
Notice that the number of relations of $P$ is $|T_P|=\sum\limits_{i\in [n]}|i^\downarrow|-n$.

 The following is one of the main definition of this article.

\begin{dfn} \label{complete flag}
   A $P$-\emph{flag} in $V$
is an $n$-tuple $(V_1,\ldots,V_n)$ of vector subspaces of $V$ which satisfies the following condition:
$$\dim\left(\sum \limits_{i\in I} V_i\right)=\left|I^\downarrow\right|,$$ for every $I\subseteq [n]$.
The set of  $P$-flags of $V$ is denoted by $\pfl(\F)$.
\end{dfn}\noindent We call \emph{standard} $P$-\emph{flag} of $V$ the tuple $$F_e^P:=(V_{1^\downarrow},\ldots,V_{n^\downarrow}).$$

A $c_n$-flag is a complete flag in the usual meaning. On the other hand, a $t_n$-flag is an $n$-tuple of lines in $\F^n$ whose generators are linearly independent. The following example shows an intermediate case between the previous ones.

\begin{ex}
  Let $V=\F^6$ and $P\in \pos(6)$ be the poset in the figure below:

   \begin{center}
 \begin{tikzpicture}[scale=.4]
  \node (one) at (2,4) {$6$};
  \node (a) at (-2,2) {$4$};
  \node (b) at (2,2) {$5$};
  \node (c) at (0,0) {$3$};
  \node (d) at (2,-2) {$2$};
  \node (e) at (-2,-2) {$1$};
  \draw (d) -- (c) -- (e);
  \draw (a) -- (c) -- (b) -- (one);
\end{tikzpicture}
  \end{center}

  Let us consider the following vector subspaces of $V$:
  \begin{itemize}
    \item $W_1:=\spn_{\F}\left\{e_1\right\}$, $W_2:=\spn_{\F}\left\{e_2\right\}$,
    \item $W_3:=\spn_{\F}\left\{e_1,e_2,e_3\right\}$, $W_4:=\spn_{\F}\left\{e_1,e_2,e_3,e_4\right\}$,
    \item $W_5:=\spn_{\F}\left\{e_1,e_2,e_3,e_5\right\}$, $W_6:=\spn_{\F}\left\{e_1,e_2,e_3,e_5,e_6\right\}$.
  \end{itemize}

  Then $(W_1,W_2,W_3,W_4,W_5,W_6)$ is the standard $P$-flag.

  The tuples $(W_1,W_2,W_3,W_4,W_6,W_5)$ and $(W_1,W_1,W_3,W_4,W_5,W_6)$ are not $P$-flags.
  Examples of $P$-flags are $$\mbox{$(W_1,W_2,W_3,W_5,W_4,W_6)$ and $(W_2,W_1,W_3,W_5,W_4,W_6)$.}$$

\end{ex}

Recall that for $F\in \pfl(\F)$, $F_i$ is the projection on the $i$-th factor.
The following proposition states some properties of a $P$-flag.

\begin{prop} \label{proprieta bandiere}
  Let $F \in \pfl(\F)$. Then
\begin{enumerate}
  \item $\dim(F_i)=|i^\downarrow|$, for all $i\in [n]$;
  \item $\dim(F_i \cap F_j)=|i^\downarrow\cap j^\downarrow|$;
  \item $F_i \subseteq F_j$ if and only if $i\leqslant_P j$;
  \item $\sum \limits_{i \in [n]} F_i = V$.
\end{enumerate}
\end{prop}
\begin{proof}
Properties $1.$ and $4.$ are obtained by Definition \ref{complete flag}, taking $I=\left\{i\right\}$ and $I=[n]$, respectively.

By the Grassmann formula and Property $1.$, $\dim(F_i \cap F_j)=\dim(F_i)+\dim(F_j)-\dim(F_i+F_j)=|i^\downarrow|+|j^\downarrow|-|i^\downarrow\cup j^\downarrow|=|i^\downarrow\cap j^\downarrow|$.

To prove Property $3.$, let $F_i \subseteq F_j$. This holds if and only if $\dim(F_i \cap F_j)=\dim(F_i)$.
But this is equivalent to $|i^\downarrow\cap j^\downarrow|=|i^\downarrow|$, which is equivalent to $i^\downarrow \subseteq j^\downarrow$, i.e.\ $i \leqslant_P j$.
\end{proof}


\begin{oss}
 Let $F \in \pfl(\F)$. Note that, by Property $3.$ of Proposition \ref{proprieta bandiere}, $F_i=F_j$ if and only if $i=j$.
\end{oss}

\begin{oss}
Let $F:=(W_1,\ldots,W_n)\in \pfl(\F)$ and $\sigma \in S_n$ such that $\sigma F:= (W_{\sigma^{-1}(1)},\ldots,W_{\sigma^{-1}(n)}) \in \pfl(\F)$.
Then $\sigma^{-1}(1) \prec \ldots \prec \sigma^{-1}(n)$ is a linear extension of $P$. In fact let $i <_P j$, $\sigma(i)=: h$ and
$\sigma(j)=: k$. Then, by Proposition \ref{proprieta bandiere}, $W_i \subseteq W_j$. Since $(\sigma F)_h=W_i$ and $(\sigma F)_k=W_j$,
we have that $h \leqslant_P k$ and this implies $h<k$, so $\sigma^{-1}(h) \prec \sigma^{-1}(k)$.
It is straightforward to check that in general the converse does not hold.
\end{oss}


We are going to prove that the set of  $P$-flags admits a structure of homogeneous space. To do this, we need the following function.

\begin{dfn}\label{fonderflas}
The Fon-Der-Flaass action (see \cite{orbit order ideal}) is the invertible function $\Psi_P: \JP \rightarrow \JP$ defined by
$$\Psi_P(I):=\left[{\min}_P([n]\setminus I)\right]^\downarrow,$$ for all $I\in \JP$.
\end{dfn}

\noindent Notice that $\Psi_P(\varnothing)=\mathrm{min}(P)$ and $\Psi_P(P)=\varnothing$. 
Now we are ready to prove one of the main results of this section.

\begin{prop} \label{prop base}
  Let $F \in \pfl(\F)$. Then there exists a basis $B:=\left\{v_1,\ldots,v_n\right\}$ of $V$ such that $F_i \cap B=\left\{v_j \in B: j \in i^\downarrow\right\}$, for all $i\in [n]$.
\end{prop}
\begin{proof}
  For any $k \geq 1$, define the induced subposet $$P_k:=\bigcup\limits_{j=1}^k\Psi_P^j(\varnothing)$$ and consider the vector space $W_k:=\sum\limits_{i\in \max(P_k)}F_i$. It is clear that there exists $k\in \N$ such that $P_k=P$. If $P_k=\left\{i_1,\cdots,i_{|P_k|}\right\}$, being $i_1<\ldots<i_{|P_k|}$, then $\left(F_{i_1},\ldots,F_{i_{|P_k|}} \right)\in \mathrm{Fl}_{P_k}(\F)$, since the order ideals of $P_k$ are order ideals of $P$.

 We construct the basis $B$ by induction on $k$. Let $k=1$. Then $P_1=\min(P)$ and $F_i=\spn_{\F}\left\{v_i\right\}$ for some $v_i \in V$, for all $i\in \min(P)$. Since $\dim \left(\sum \limits_{i\in P_1}F_i\right)=|P_1|$, we have that $|\left\{v_i:i\in P_1\right\}|=|P_1|$ and the elements $v_1,\ldots,v_{|P_1|}$ are linearly independent.
 We let $B^1:=\left\{v_i:i\in P_1\right\}$. Then $F_i \cap B^1=\left\{v_i\right\}=\left\{v_j \in B^1:j\in i^\downarrow\right\}$, for all $i\in P_1$.

Now let $k>1$. By induction, we have a basis $B^{k-1}$ of $W_{k-1}$ such that  $F_i \cap B^{k-1}=\left\{v_j \in B^{k-1}:j\in i^\downarrow \right\}$, for all $i\in P_{k-1}=P_k \setminus \max(P_k)$.

Let $\max(P_k)=\left\{p_1,\ldots,p_r\right\}$; by Proposition \ref{proprieta bandiere}, $F_q \subseteq F_{p_i}$ for all $q \vartriangleleft p_i$, $i\in [r]$, and
$$\dim\left(\sum\limits_{q\vartriangleleft p_i}F_q\right)=\left|\left\{q \in P_k: q \vartriangleleft p_i\right\}^\downarrow\right|=|p_i^\downarrow|-1=\dim(F_{p_i})-1.$$
 This implies the existence of an element $v_{p_i}\in F_{p_i}\setminus \left(\sum\limits_{q\vartriangleleft p_i}F_q\right)$, for all $i\in [r]$. We let $$B^k:=B^{k-1} \cup \left\{v_{p_1},\ldots,v_{p_r}\right\}.$$ 

 It remains to prove that $B^k$ is a basis of $W_k$. Let $i\in [r]$ and assume by contradiction $v_{p_i} \in \sum\limits_{j\in P_k \setminus \left\{p_i\right\}} F_j$. Then $F_{p_i} \subseteq \sum \limits_{j \in P_k \setminus \left\{p_i\right\}}F_j$. Hence $$|P_k|=\dim\left(\sum \limits_{j\in P_k}F_j\right)=\dim\left(\sum \limits_{j \in P_k \setminus \left\{p_i\right\}}F_j\right)=|P_k|-1.$$
If $P_k=P$, we let $B:=B^k$. Then $B$ is a basis of $V$ with the stated property.
\end{proof}

\begin{cor}
  Let $F \in \pfl(\F)$. Then the set $\left\{F_i: i \in [n]\right\}$ generates, by sums and intersections, a distributive lattice isomorphic to $\mathcal{J}(P)$.
  Moreover $$\dim\left(\bigcap \limits_{i\in I} F_i\right)=\bigg|\bigcap\limits_{i\in I}i^\downarrow\bigg|,$$ for all $I\subseteq [n]$.
\end{cor}
\begin{proof}
 By Proposition \ref{prop base}, the lattice generated by $\left\{F_i: i \in [n]\right\}$ is isomorphic to the lattice $L$ generated by $\left\{F_i \cap B : i \in [n]\right\}$, with respect to the operations $\cup$
 and $\cap$, which is distributive. From this we deduce also the last assertion. Moreover, by construction $L$ is isomorphic to $\mathcal{J}(P)$.
\end{proof}

Let $F\in \pfl(\F)$.
If $B:=\left\{v_1,\ldots,v_n\right\}$ is a basis of $V$ such that $F_i \cap B=\left\{v_j \in B: j \in i^\downarrow\right\}$, for all $i\in [n]$, we say that
$B$ is $F$-\emph{adapted}.
Choosing an $F$-adapted basis $\left\{v_1,\ldots,v_n\right\}$ of $V$ for any $P$-flag $F\in \pfl(\F)$, we can define
a function $\beta :  \pfl(\F) \rightarrow \gl(n,\F)$ by setting $\beta(F)$ as the unique matrix which satisfies $\beta(F)e_i=v_i$, for all $i\in [n]$.

\begin{thm} \label{fl omogeneo} Let $\pi : \gl(n,\F)\rightarrow \gl(n,\F)/I^*(P;\F)$ be the canonical projection.
 Then the function $$\pi \circ \beta: \pfl(\F) \rightarrow \gl(n,\F)/I^*(P;\F)$$ is bijective.
\end{thm}
\begin{proof}
 An action of the group $\gl(n,\F)$ on $\pfl(\F)$ is given by  $$(AF)_i:=AF_i,$$ for all $i\in [n]$, $A\in \gl(n,\F)$, and $F\in \pfl(\F)$.  In fact dimensions are preserved and $A\left(\sum \limits_{i \in I} F_i \right)=\sum \limits_{i\in I} AF_i$
 for all $I\subseteq [n]$, $A\in \gl(n,\F)$. Since $\beta(F)F_e^P=F$, for all $F \in \pfl(\F)$, this action is transitive and $AF_e^P=F_e^P$
 if and only $A\in I^*(P;\F)$, so the result follows.
\end{proof}
For arbitrary fields, we call $\pfl(\F)$ a $P$-\emph{flag space}.
The set $\pfl(\mathbb{R})$ turns out to have a structure of differentiable manifold, which we call $P$-\emph{flag manifold}. We recover the real flag manifold for $P=c_n$.

\begin{cor} \label{dimensione}
  Let $P\in \pos(n)$. The set $\pfl(\mathbb{R})$ is a differentiable manifold of dimension $n(n-1)-|T_P|$.
\end{cor}
\begin{proof}
  Notice that $I^*(P;\mathbb{R})$ is a closed subgroup of the Lie group $\gl(n;\mathbb{R})$; in fact an incidence group is defined by the vanishing of suitable entries, depending on $P$. By the closed-subgroup theorem (see, e.g. \cite[Theorem~9.3.7]{Structure and geometry of Lie groups}), $I^*(P;\mathbb{R})$ is a Lie subgroup and, by \cite[Theorem~10.1.10]{Structure and geometry of Lie groups}, the quotient $\gl(n,\mathbb{R})/I^*(P;\mathbb{R})$ has a unique real manifold structure.

 Since the Lie algebra of $I^*(P;\mathbb{R})$ is the Lie algebra of the incidence algebra $I(P;\mathbb{R})$ and its dimension is $|P|+|T_P|$, we obtain the stated formula (see, e.g.\cite[Corollary~10.1.12]{Structure and geometry of Lie groups}).
\end{proof}


\begin{oss}
  By Theorem \ref{fl omogeneo}, there exists a canonical projection $\pfl(\C)\rightarrow \mathrm{Fl}_n(\C)$ whose fibers
  are affine spaces of dimension $\left|[n]^2_< \setminus T_P\right|$. It follows that this projection is a homotopy equivalence.
\end{oss}

By Theorem \ref{fl omogeneo} we can deduce the cardinality of the set of $P$-flags on a finite field of $q$ elements.
\begin{cor} \label{cardinalita bandiere} Let $P\in \pos(n)$. Then
  $$|\pfl(\F_q)|=q^{\frac{n(n-1)}{2}-|T_P|}[n]_q!.$$
\end{cor}
\begin{proof}
  First of all recall the well-known formula $$|\gl(n;\F_q)|={\prod \limits_{i=0}^{n-1}(q^n-q^i)}=q^{\frac{n(n-1)}{2}}\prod \limits_{i=1}^{n}(q^i-1).$$ It is clear that
  $|I^*(P;\F_q)|=(q-1)^nq^{|T_P|}$. Then the result follows from Theorem \ref{fl omogeneo}.
\end{proof}

The following proposition reveals a duality phenomenon, which does not appear in the classical case, since a chain $c_n$ is self-dual (see Definition \ref{dual}).

\begin{prop} \label{dualita duale}
  Let $P\in \pos(n)$. Then we have a bijection
  $$\mathrm{Fl}_P^*: \mathrm{Fl}_{\mathrm{P}}(\F) \rightarrow \mathrm{Fl}_{\mathrm{P}^*}(\F)$$ defined by setting
  $$\mathrm{Fl}_P^*(F)_i=\spn_{\F}\left\{v_{n+1-j}: j \in i_{P^*}^\downarrow\right\},$$
  for all $i\in [n]$, $F \in \pfl(\F)$, where $\left\{v_1,\ldots,v_n\right\}$ is an $F$-adapted basis of $V$.

\end{prop}
\begin{proof}

Let $F \in \pfl(\F)$ and $\left\{v_1,\ldots,v_n\right\}$ be an $F$-adapted basis of $V$.
 Let $w_i:=v_{n+1-i}$, for all $i\in [n]$; therefore, by Definition \ref{dual}, $\left\{w_1,\ldots,w_n\right\}$ is an $\mathrm{Fl}^*(F)$-adapted basis of $V$.
It is clear by construction that $\mathrm{Fl}_{P^*}^* \circ \mathrm{Fl}_P^*$  and $\mathrm{Fl}_P^*  \circ \mathrm{Fl}_{P^*}^*$
 are the identity on $\pfl(\F)$ and $\mathrm{Fl}_{\mathrm{P}^*}(\F)$, respectively.
\end{proof}

In the example below we present in a particular case the duality in Proposition \ref{dualita duale}.

\begin{ex} \label{esempio poset V} Given a positive integer $n$, the $n$-th \emph{configuration space} of a set $X$ is
$$\cs_n[X]:=\left\{(x_1,\ldots,x_n)\in X^n: i \neq j \Rightarrow x_i \neq x_j \right\}.$$
Unless otherwise specified, the symbol $\simeq$ stands for a bijection.

  Let $P\in \pos(3)$ be the poset whose Hasse diagram is the one on the left in the following figure.
  The Hasse diagram on the right is the one of $P^*$.
\begin{center}
 \begin{tikzpicture}[scale=.7]
  \node (one) at (0,2) {$3$};
  \node (b) at (-1,0) {$1$};
  \node (c) at (1,0) {$2$};
  \draw (one) -- (b) -- (one) -- (c);
\end{tikzpicture} $\,\,\,\,\,\,\,\,\,\,$
\begin{tikzpicture}[scale=.7]
  \node (a) at (0,0) {$1$};
  \node (b) at (-1,2) {$2$};
  \node (c) at (1,2) {$3$};
  \draw (a) -- (b)  (a) -- (c);
\end{tikzpicture}
  \end{center}
  Let $\left\{v_1,v_2,v_3\right\}$ be a basis of $V$, $V_1 := \spn_{\F}\left\{v_1\right\}$, $V_2:=\spn_{\F}\left\{v_2\right\}$
  and $V_3:=V$. Then $F:=(V_1,V_2,V_3) \in \pfl(\F)$  and
  $$\mathrm{Fl}_P^* (F)=\left(\spn_{\F}\left\{v_3\right\},\spn_{\F}\left\{v_2,v_3\right\}, \spn_{\F}\left\{v_1,v_3\right\}\right). $$
  Moreover, it is immediate to check that $\pfl(\F)\simeq \cs_2\left[\mathbb{P}(\F^3)\right]$ and $\mathrm{Fl}_{\mathrm{P}^*}(\F)\simeq \cs_2\left[\gr_{\F}(2,3)\right]\simeq \cs_2\left[\mathbb{P}(\F^3)\right]$.
\end{ex}

%
%
%

\begin{oss} \label{oss spanning configurations}
  We observe that $\pfl(\C)$ is a subset of the moduli space of spanning configurations $X_{\alpha,n}$, introduced in \cite{Spanning subspace configurations},
with $\alpha=(|1^\downarrow|,\ldots,|n^\downarrow|)$.
Moreover $\mathrm{Fl}_{t_n}(\C)=X_{1^n,n}$, where $1^n=(1,1,\ldots,1)\in [1]^n$.
Notice that $\mathrm{Fl}_{t_n}(\C)$ is also the moduli space $X_{n,n}$ of $n$ independent lines in $\C^n$ of \cite{A flag variety for the Delta Conjecture}.
\end{oss}

\subsection{$(Q,P)$-cells} \label{QPcelle}

In this section we consider the left action of the incidence group $I^*(Q;\F)$ on  $\pfl(\F)$, where $P,Q\in \pos(n)$. For $P=Q=c_n$, the orbits of this action are the classical Schubert cells of the flag variety, which are indexed by the elements of the symmetric group $S_n$. In Proposition \ref{prop orbite} we prove that for other choices of $Q$ and $P$, the action of $I^*(Q;\mathbb{R})$ on $\pfl(\R)$ has infinitely many orbits. For other general results on infiniteness of double quotients see for instance \cite{Duck}, \cite{Gandini} and references therein.
Nevertheless we consider a finite subset of these orbits, corresponding to permutations in $S_n$, which have a particularly nice description as in the classical case.

\begin{prop} \label{prop orbite}
 The double quotient $I^*(Q;\mathbb{R})\backslash \gl(n;\mathbb{R})\slash I^*(P;\mathbb{R})$ is finite if and only if
 $P=Q=c_n$.
\end{prop}
\begin{proof}
 It is well known that if $P=Q=c_n$ then the double quotient considered is in bijection with the symmetric group $S_n$.

 Let $Q$ be any poset and $P\neq c_n$.
 The maximal possible dimension $d$ of an orbit of $I^*(Q;\mathbb{R})$ is reached when $Q=c_n$ and the isotropy group is the group of invertible diagonal matrices $I^*(t_n;\F)$; then   $d=\dim(I^*(c_n;\mathbb{R}))-n=\frac{n(n-1)}{2}$ by \cite[Corollary~10.1.12]{Structure and geometry of Lie groups}. By Corollary \ref{dimensione}, $\dim(\gl(n;\mathbb{R})\slash I^*(P;\mathbb{R}))=n(n-1)-|T_P|$. Since $P\neq c_n$, the minimum of $n(n-1)-|T_P|$ is reached when $P$ has exactly two incomparable elements; its value is $n(n-1)-\frac{n(n-1)-2}{2}=\frac{n(n-1)}{2}+1$.
 Therefore $\dim(\gl(n;\mathbb{R})\slash I^*(P;\mathbb{R}))$ is always strictly greater than the dimension of every orbit of $I^*(Q;\mathbb{R})$,
 which implies the infiniteness of the set of such orbits.
\end{proof}
Now we consider a collection of orbits of $I^*(Q;\F)$ on $\pfl(\F)$ which share some properties with the classical Schubert cells of the flag variety.
For any permutation $\sigma \in S_n$, let us define the $P$-flag

$$F_\sigma^P:=\left(\spn_{\F}\left\{e_{\sigma(i)}:i\in 1^\downarrow_P\right\},\ldots,\spn_{\F}\left\{e_{\sigma(i)}:i\in n^\downarrow_P\right\}\right).$$
When $\sigma$ is the identity we recover the standard $P$-flag $F_e^P$.

\begin{dfn}
The $(Q,P)$-cell in $\pfl(\F)$ corresponding to $\sigma \in S_n$ is the orbit
$$C^{Q,P}_\sigma(\F) := \left\{AF_\sigma^P:A\in I^*(Q;\F)\right\}.$$
\end{dfn}

These cells can be described as homogeneous spaces. Before to state this result, we need some definitions.

\begin{dfn}
Let $P,Q \in \pos(n)$ and $\sigma \in S_n$. The poset $[QP]_{\sigma}:=([n],\leqslant_{Q,P,\sigma})$
is defined by setting
  $$i \leqslant_{Q,P,\sigma} j \Leftrightarrow \mbox{$i \leqslant_Q j$ and $\sigma^{-1}(i) \leqslant_P \sigma^{-1}(j)$},$$ for every $i,j\in [n]$.
\end{dfn} Notice that $[QP]_\sigma \leqslant Q$, for every $P,Q\in \pos(n)$, $\sigma \in S_n$.

\begin{ex}
Let $Q\in \pos(n)$. It is clear that $[Qc_n]_e=Q$ and $[Qt_n]_{\sigma}=t_n$ for all $\sigma\in S_n$.
Moreover $[Qc_n]_{w_0}=t_n$, where $w_0=n \cdots 321$.
\end{ex}

\begin{oss} \label{oss forest-like}
  The Hasse diagram of the poset $[c_nc_n]_{\sigma}$ is the graph $G_\sigma$ defined in \cite{Forest-like permutations}.
  This is also related to the inversion graph of the permutation $\sigma$ (see \cite{OhPostnikov}).
\end{oss}

\begin{oss}
  The induced subposet $\left\{[c_nc_n]_{\sigma}:\sigma \in S_n\right\} \subseteq \pos(n)$ is isomorphic to the dual of the right $\leqslant_R$ weak order of $S_n$.
  In fact, by \cite[Proposition~3.1.3]{BB}, $\sigma \leqslant_R \tau$ if and only if $T_L(\sigma) \subseteq T_L(\tau)$, where $T_L(\sigma)$ is the set of left inversions of $\sigma$. This is equivalent to $[c_nc_n]_\tau \leqslant [c_nc_n]_\sigma$.
\end{oss}

\begin{dfn}
Let $\sigma \in S_n$. The $(Q,P)$-{\em inversion number} $\inve_{Q,P}(\sigma)$ of $\sigma$ is defined by
\begin{eqnarray*}
  \inve_{Q,P}(\sigma) &:=& |\left\{(i,j)\in [n]^2_<: i<_Qj,\sigma^{-1}(i)\nless_P \sigma^{-1}(j)\right\}|.
\end{eqnarray*}
\end{dfn}

For $Q=P=c_n$ this function gives the usual inversion number $\mathrm{inv}(\sigma)$ of a permutation in $S_n$.

\begin{thm}\label{canonical} Let $P,Q\in \pos(n)$ and $\sigma\in S_n$. Then we have the following bijections:
  $$C^{Q,P}_\sigma(\F) \simeq I^*(Q;\F)/I^*([QP]_\sigma;\F) \simeq \F^{\inve_{Q,P}(\sigma)}.$$
\end{thm}
\begin{proof}
  Let $F^P_{\sigma}=(V_1,\ldots,V_n)$, where $V_j=\spn_{\F}\left\{e_{\sigma(i)}:i\in j^\downarrow_P\right\}$, for all $j\in [n]$.
  Let $A\in I^*(Q;\F)$ be an element of the isotropy group of $F^P_{\sigma}$ under the action $AF^P_{\sigma}=(AV_1,\ldots,AV_n)$.
  We prove that $A \in I^*([QP]_\sigma;\F)$.

  We have that $V_1=\spn_{\F}\left\{e_{\sigma(1)}\right\}$ and $AV_1=V_1$ implies $A_{i,\sigma(1)}=0$ for all $i<_Q \sigma(1)$. Again $AV_2=V_2$ implies
 $A_{i,\sigma(2)}=0$ for all $i<_Q \sigma(2)$ such that $i\not \in \left\{\sigma(k):k\in 2^\downarrow_P\right\}$.
In general, $AV_j=V_j$ implies $A_{i,\sigma(j)}=0$ for all $i<_Q \sigma(j)$ such that $i\not \in \left\{\sigma(k):k\in j^\downarrow_P\right\}$.
Therefore the isotropy group of $F^P_\sigma$ is contained in the set

$$\bigcap \limits_{j=1}^n\left\{A \in I^*(Q;\F): A_{i,\sigma(j)}=0, \, \forall \, i\not \in \left\{\sigma(k):k\in j^\downarrow_P\right\}\right\}=$$
$$\bigcap \limits_{j=1}^n\left\{A \in I^*(Q;\F): A_{i,j}=0, \, \forall \, i\not \in \left\{k:\sigma^{-1}(k)\in [\sigma^{-1}(j)]^\downarrow_P\right\}\right\} =I^*([QP]_\sigma;\F).$$
By definition of $F^P_\sigma$ and $[QP]_\sigma$, it follows that $I^*([QP]_\sigma;\F)$ is contained in the isotropy group of $F^P_\sigma$, and the first bijection is proved.

A coset of $A\in I^*(Q;\F)$ is determined setting $A_{i,i}=1$ for all $i\in [n]$ and $A_{ij}=0$ whenever $(i,j)\in T_{[QP]_{\sigma}}$.
Since $|T_Q\setminus T_{[QP]_{\sigma}}|=\inve_{Q,P}(\sigma)$, the second bijection follows.
  \end{proof}

Immediate consequences of Theorem \ref{canonical} are the following statements.

\begin{cor}
  Let $\F_q$ be a finite field. Then $$|C^{Q,P}_\sigma(\F_q)|=q^{\inve_{Q,P}(\sigma)},$$ for all $\sigma \in S_n$.
\end{cor}

A poset is said to be \emph{strict Sperner} if it is a graded poset in which all maximum antichains are rank levels.
The next result gives a bijection between a $(Q,P)$-cell $C^{Q,P}_\sigma(\F)$ and the derived algebra of the Lie algebra $I([QP]^c_{\sigma}(Q);\F)$, whenever $P$ is a strict Sperner poset.
\begin{cor} \label{corollario unipotente}
  If $P$ is strict Sperner, then we have a bijection $$C^{Q,P}_\sigma(\F) \simeq U([QP]^c_{\sigma}(Q);\F),$$
  for all $\sigma \in S_n$.
\end{cor}
\begin{proof}
  By definition, if $P$ is strict Sperner then the poset $[QP]_{\sigma}$ is complemented in $Q$. In fact, in a strict Sperner poset, the relation $\nleqslant$ is transitive. Then $T_Q\setminus T_{[QP]_{\sigma}}=T_{[QP]^c_{\sigma}(Q)}$. Hence the result follows by Proposition \ref{Prop complementato}.
\end{proof}

\section{Incidence stratifications} \label{sezione grassmanniane}

In this section we provide a partition of the projective space $\mathbb{P}(\F^n)$, induced by the action of the incidence group $I^*(Q;\F)$, for any poset $Q \in \pos(n)$.
The orbits of such an action turn out to be in one-to-one correspondence with the elements of the distributive lattice $\JQ$.

This decomposition induces a partition of any subset of a projective space. We investigate the induced partition on Grassmannian varieties, recovering the Schubert cell partition, for $Q=c_n$, and the matroid strata
introduced in \cite{Gelfand}, for $Q=t_n$.

\subsection{$Q$-stratification of a projective space}

Let $Q\in \pos(n)$, $V=\F^n$ and $\mathbb{P}(V)$ its projective space.
The subalgebra $I(Q;\F) \subseteq \End(V)$ has invariant-subspace lattice isomorphic to $\mathcal{J}(Q)$, where
$I(Q;\F)$ acts on the elements of $V$ by left multiplication.
\begin{oss}
  The socle filtration of the action of $I(Q;\F)$ on $V$ is given by $$\soc^i(Q)\simeq\bigoplus\limits_{j\in \max\left[\Psi_Q^i(\varnothing)\right]}\spn_{\F}\left\{e_j\right\},$$ for all $i>0$ such that $\Psi^i_Q(\varnothing) \subsetneq \Psi^{i+1}_Q(\varnothing)$, where $\Psi_Q$ is the function of Definition \ref{fonderflas}.
\end{oss}

Clearly this action carries an action of $I^*(Q,\mathbb{F})$ on $\mathbb{P}(V)$, whose orbits are described in the following theorem.
Recall that $V_I:=\spn_{\F}\left\{e_i:i\in I\right\}$, for any subset $I\subseteq [n]$.

\begin{thm}\label{pspazio}
  An orbit of the action of $I^*(Q;\F)$ on $\mathbb{P}(V)$ is of the form
  $$Q_I(\mathbb{F}):=\mathbb{P}(V_I)\setminus \bigcup\limits_{i \in \mathrm{max}(I)}\mathbb{P}\left(V_{I \setminus \{i\}}\right),$$
   for any $I\in \JQ\setminus\{\varnothing\}$ and the collection of {cells}\footnote{The use of the word \emph{cell} in this article  does not refer in general to affine spaces.} $\{Q_I(\F):I\in \JQ\setminus\{\varnothing\}\}$ is a
    partition of $\mathbb{P}(V)$. 
 The Zariski closure of $Q_I(\C)$ is given by $$\overline{Q_I(\C)}=\biguplus\limits_{H \in \mathcal{J}(I)\setminus \{\varnothing\}} Q_H(\C)=\mathbb{P}\left(V_I\right),$$ for all $I\in \JQ$.
\end{thm}
\begin{proof}
Let $v\in V$ be expressed as $v=a_1e_{i_1}+\ldots+a_ke_{i_k}$ for some $k\in [n]$, $a_1,\ldots,a_k\in \F\setminus \{0\}$.
Let $M:=\max_Q\{i_1,\ldots,i_k\}$ and $I:=M^\downarrow \in \JQ$. Then $v$ lies in $V_I \setminus \bigcup\limits_{i \in M}V_{I \setminus \{i\}}$.
Since the action of $I^*(Q;\F)$ on $V_I \setminus \bigcup\limits_{i \in M}V_{I \setminus \{i\}}$ is transitive and
the projection of this set on $\mathbb{P}(V)$ is $\mathbb{P}(V_I)\setminus \bigcup\limits_{i \in M}\mathbb{P}\left(V_{I \setminus \{i\}}\right)$,
the first assertion follows.
Finally we have that $\overline{Q_I(\C)}=\mathbb{P}(V_I)$; since $V_I$ is $I^*(Q;\mathbb{C})$-invariant, the last assertion can be deduced by repeating the previous arguments to the projective space $\mathbb{P}(V_I)$.
\end{proof} In analogy with the case $\F=\C$ in Theorem \ref{pspazio}, for any field $\F$, we say that $Q_I(\F)$ is a $Q$-\emph{Schubert cell} of $\mathbb{P}(V)$ and we define $\overline{Q_I(\F)}:=\bigcup\limits_{H \in \mathcal{J}(I)\setminus \{\varnothing\}} Q_H(\F)$,
saying that $\overline{Q_I(\F)}$ is a $Q$-\emph{Schubert variety} of $\mathbb{P}(V)$, which turns out to be a projective space.

The following are immediate consequences of Theorem \ref{pspazio}.

\begin{cor} \label{corollario dimensione pspazio} Let $Q\in \pos(n)$ and $I\in \JQ\setminus \{\varnothing\}$. Then $$\mathrm{dim}(\overline{Q_I(\mathbb{C}}))=|I|-1.$$
\end{cor}

\begin{cor}
  The poset of $Q$-Schubert varieties of $\mathbb{P}(V)$, ordered by inclusion, is isomorphic to $\JQ\setminus \{\varnothing\}$.
  Moreover, if $I\cap J \neq \varnothing$ then
$$\overline{Q_I(\F)} \cap \overline{Q_J(\F)}=\overline{Q_{I \cap J}(\F)}.$$
\end{cor}

In the case of a finite field $\mathbb{F}_q$, we provide a formula for the number of points of a $Q$-Schubert cell $Q_I(\mathbb{F}_q)$.

\begin{cor} Let $\F_q$ be a finite field. Then
  $$|Q_I(\F_q)|=\sum\limits_{\substack{H \in \mathcal{J}(I)\setminus \{\varnothing\}\\I\setminus H \subseteq \max(I)}}(-1)^{|I\setminus H|} [|H|]_q.$$
\end{cor}
\begin{proof}
  By Theorem \ref{pspazio} we know that $\mathbb{P}(V_I)=\biguplus \limits_{H \in \mathcal{J}(I)\setminus \{\varnothing\}}Q_H(\mathbb{F}_q)$.
  It is known (see \cite[Example~3.9.6]{StaEC1}) that the  M\"obius function of a distributive lattice is
  $$\mu(H,I)=\left\{
               \begin{array}{ll}
                 (-1)^{|I\setminus H|}, & \hbox{if $I\setminus H\subseteq \max(I) $;} \\
                 0, & \hbox{otherwise.}
               \end{array}
             \right.$$
  Since $|\mathbb{P}(\mathbb{F}^{n}_q)|=[n]_q$,
  we obtain our formula by  M\"obius inversion.
  \end{proof}

With the following definition we introduce a general procedure to decompose subsets of projective spaces. In the subsequent sections we
apply this approach to Grassmannians and $P$-flag spaces.

\begin{dfn} \label{def incidence strat}
 Let $X \subseteq \mathbb{P}(\F^n)$.
 Given a poset $Q\in \pos(n)$, we call \emph{incidence stratification} of $X$ the set
 $$\left\{\overline{Q_I(\F)}\cap X : I \in \mathcal{J}(Q) \right\} \setminus \left\{\varnothing \right\}.$$
\end{dfn}

\subsection{$Q$-stratification of a Grassmannian}

Let $n>0$, $k\in [n]$ and $Q\in \pos(n)$. We need to define a suitable poset $Q^k_<$ in order to realize an incidence stratification
of the Grassmannian $\gr_{\F}(k,n)$, generalizing Schubert varieties and matroidal strata.

Consider the Cartesian $k$-th power $Q^k$ of the poset $Q$.
Recall that the order on $Q^k$ is defined by $$i \leqslant_{Q^k} j \Leftrightarrow i_h \leqslant_Q j_h,  \text{  for every  } h\in [k],$$
for all $i,j \in [n]^k$, where $i_h$ is the projection of $i$ on the $h$-th component.

The poset $Q^k$ admits an action of the symmetric group $S_k$, as showed in the next proposition, whose proof is straightforward.

\begin{prop} \label{prop az automorf}
  Let $\sigma \in S_k$. Then the action on $[n]^k$ defined by $$\sigma i:= (i_{\sigma^{-1}(1)},\ldots,i_{\sigma^{-1}(k)}),$$ for all $i\in [n]^k$, is an automorphism of the poset $Q^k$. This defines a group morphism $S_k \rightarrow \aut(Q^k)$.
\end{prop}

The following poset is fundamental for our constructions.

\begin{dfn} \label{ordine k}
  The poset $Q^k_< :=\left([n]^k_<, \preccurlyeq_{Q^k}\right)$ is defined by letting
  $$i \preccurlyeq_{Q^k} j \Leftrightarrow \sigma i\leqslant_{Q^k} j,$$ for some $\sigma \in S_k$, for all $i,j \in [n]^k_<$.
\end{dfn} Notice that $Q^1_<=Q$. For $k>1$ it could be not obvious that $Q^k_<$ is a poset. This follows from Proposition \ref{prop az automorf}, as we are going to show.
Let $i,j,h \in [n]^k_<$.
\begin{enumerate}
  \item reflexivity: straightforward, by taking $\sigma=e$.
  \item antisymmetry: let $\sigma i\leqslant_{Q^k} j$ and $\tau j\leqslant_{Q^k} i$, for some $\sigma,\tau \in S_k$. Then $\tau\sigma i \leqslant_{Q^k} i$. From the fact that $i_1<\ldots<i_k$, we obtain $\tau\sigma=e$. Hence $i \leqslant_{Q^k} \tau j \leqslant_{Q^k} i$, which implies $\tau=\sigma=e$ and $i =j$.
  \item transitivity: let $h \preccurlyeq_{Q^k} i$ and $i \preccurlyeq_{Q^k} j$; then there exist $\sigma, \tau \in S_k$ such that $\sigma h \leqslant_{Q^k} i \leqslant_{Q^k}  \tau j$. This implies
  $\tau^{-1}\sigma h \leqslant_{Q^k} j$.
\end{enumerate} It is clear that $i \leqslant_{Q^k} j$ implies $i \preccurlyeq_{Q^k} j$, i.e.\ the poset $Q^k_<$ is a refinement of $([n]_<^k,\leqslant_{Q^k})$, the induced subposet of $Q^k$. If $Q=c_n$, they are actually the same poset, as stated in the following proposition.

\begin{prop}\label{refinem}
  Let $n \geqslant 1$ and $k \in [n]$. Then  $(c_n)^k_<=([n]_<^k,\leqslant_{c_n^k})$.
\end{prop}
\begin{proof}

Let $i,j \in [n]^k_<$ with  $i \nleqslant_{c_n^k} j$. Then there exists a minimal $h\in [k]$ such that
 $j_h<i_h$. If $h=k$ then it is immediate to check that $\sigma i \nleqslant_{c_n^k} j$, for all $\sigma \in S_k$. Let $h<k$ and $\sigma\in S_k$. There are three cases to be considered.
 \begin{enumerate}
   \item $\sigma^{-1}(h)=h$: we have that $j_h<i_h=i_{\sigma^{-1}(h)}$ and this implies $\sigma i \nleqslant_{c_n^k} j$.
   \item $\sigma^{-1}(h)>h$: in this case $i_{\sigma^{-1}(h)}>i_h>j_h$, so $\sigma i \nleqslant_{c_n^k} j$.
   \item $\sigma^{-1}(h)<h$: in this case $h>1$. There exists $t\in [h-1]$ such that $\sigma^{-1}(t)\geqslant h$; then $i_{\sigma^{-1}(t)} \geqslant i_h > j_h>j_t$ and $\sigma i \nleqslant_{c_n^k} j$.
 \end{enumerate} Then  $i \preccurlyeq_{c_n^k} j$ implies  $i \leqslant_{c_n^k} j$. \end{proof}

We can consider $Q^k_<$ as an element of $\pos\left(\binom{n}{k}\right)$; in fact the lexicographic order on $[n]^k_<$ provides a natural labeling of $Q^k_<$, as showed in the next proposition.
\begin{prop} \label{proposizione lex}
 Let $Q\in \pos(n)$ and $k\geqslant 1$. Then $a \preccurlyeq_{Q^k} b \Rightarrow a \leqslant_{\mathrm{lex}} b$,
 for all $a,b\in [n]^k_<$.
\end{prop}
\begin{proof}

  We claim that $Q \leqslant P$ implies $Q_<^k \hookrightarrow P^k_<$, for all $Q,P \in \pos(n)$. In fact,  $\sigma a \leqslant_{Q^k} b$ implies
  $\sigma a \leqslant_{P^k} b$, for all $a,b \in [n]^k_<$, $\sigma \in S_k$.
  Since $Q \leqslant c_n$, we obtain $Q_<^k \hookrightarrow (c_n)^k_<$.
   By Proposition \ref{refinem}, $(c_n)^k_<=([n]^k_<,\preccurlyeq_{c^k_n})=([n]^k_<,\leqslant_{c^k_n})$. Moreover we have that $([n]^k_<,\leqslant_{c^k_n}) \hookrightarrow ([n]^k_<,\leqslant_{\mathrm{lex}})$ is a linear extension of $(c_n)^k_<$. Then $Q^k_< \hookrightarrow (c_n)^k_< \hookrightarrow \left([n]^k_<,\leqslant_{\mathrm{lex}}\right)\simeq c_{\binom{n}{k}}$ gives a linear extension of $Q^k_<$.
\end{proof}

The duality proved in the following proposition is a poset theoretic version of the Grassmannian duality $\gr_{\F}(k,n) \simeq \gr_{\F}(n-k,n)$.

\begin{prop} \label{prop duali}
 Let $Q\in \pos(n)$. Then the following poset isomorphism holds for all $k\in [n-1]$:
 $$Q^k_< \simeq \left(Q^{n-k}_<\right)^*.$$
\end{prop}
\begin{proof}
 Let $a,b\in [n]^k_<$. Recall that we consider $\bigcup\limits_{k=0}^n[n]^k_<$ as the Boolean algebra $\mathcal{P}([n])$. We let $g^c:=[n]\setminus g \in [n]^{n-k}_<$, for all $g\in [n]^k_<$.
 We claim that $a \preccurlyeq_{Q^k} b$ if and only if $a \setminus b \preccurlyeq_{Q^h} b \setminus a$, where $h:=k-m$ and $m:=|a\cap b|$.
 If $a\cap b=\varnothing$ there is nothing to prove. Assume $a\cap b\neq\varnothing$.

 \begin{enumerate}
   \item $a \preccurlyeq_{Q^k} b \Rightarrow a \setminus b \preccurlyeq_{Q^h} b \setminus a$: by hypothesis there exists $\omega \in S_k$
   such that $a \leqslant_{Q^k} \omega  b$. Let $a_i:=z \in a \cap b$,
   $a_j:=x \leqslant_Q z =: (\omega b)_j$ and $z \leqslant_Q y =: (\omega b)_i$. Then $x \leqslant_Q y$ and
   $a \leqslant_{Q^k} (\tau\omega) b$, where, if $i\neq j$, $\tau \in S_k$ is the transposition such that $(\tau\omega b)_j=y$ and $(\tau\omega b)_i=z$,
   otherwise $\tau$ is the identity.
   We then conclude by repeated use of this argument.
   \item $a \setminus b \preccurlyeq_{Q^h} b \setminus a \Rightarrow a \preccurlyeq_{Q^k} b$: by hypothesis there exists $\omega \in S_h$
   such that $a\setminus b \leqslant_{Q^h} \omega \left(b\setminus a\right)$. Let $\sigma, \tau \in S_k$
   be the permutations such that $\sigma a =(u_1,\ldots,u_m,v_1,\ldots,v_h)$ and $\tau b =(u_1,\ldots,u_m,z_1,\ldots,z_h)$,
   where $(u_1,\ldots,u_m)=a \cap b$, $(v_1,\ldots,v_h)=a\setminus b$ and $(z_1,\ldots,z_h)=\omega \left(b\setminus a\right)$. Hence $\sigma a \leqslant_{Q^k} \tau b$
   and this implies $a \preccurlyeq_{Q^k} b$.
\end{enumerate}
Notice that $a^c \setminus b^c=b\setminus a$ and $b^c \setminus a^c =a\setminus b$;
 hence, by the previous claim we have that
\begin{eqnarray*}
   a \preccurlyeq_{Q^k} b & \Leftrightarrow&  a \setminus b \preccurlyeq_{Q^h} b \setminus a\\
   &\Leftrightarrow& b^c \setminus a^c \preccurlyeq_{Q^h} a^c \setminus b^c \\
   &\Leftrightarrow & b^c \preccurlyeq_{Q^{n-k}} a^c,
 \end{eqnarray*} where $h:=k-|a\cap b|$. \end{proof}

\begin{oss}
  By the proof of Proposition \ref{prop duali}, we know that $a \preccurlyeq_{Q^k} b$ if and only if $a \setminus b \preccurlyeq_{Q^h} b \setminus a$, for all $a,b \in [n]^k_<$,
  where $h:=k-|a\cap b|$.
  This is very useful when dealing with explicit examples of the poset $Q^k_<$.
\end{oss}

Let $Q\in \pos(n)$; there exists a
representation $\pi^k_Q:I^*(Q;\F)\rightarrow \aut\bigl(\bigwedge^kV \bigr)$ given by diagonal action:
$$A(v_1 \wedge \ldots \wedge v_k)=Av_1 \wedge \ldots \wedge Av_k,$$ for every $A\in I^*(Q;\F)$ and $v_1,\ldots,v_k \in V$.

\begin{thm} \label{teorema sottogruppo Qk}
  The group morphism $\pi^k_Q$ is injective and $\pi^k_Q(I^*(Q;\F))$ is a subgroup of the incidence group $I^*(Q^k_<;\F)$.
\end{thm}
\begin{proof}
  Let $A\in I^*(Q;\F)$ such that $\pi^k_Q(A)=\id$. Then any subspace of dimension $k$ of $V$ is $A$-invariant. This implies that $A=\id_n$.
  Moreover we have that, for $i \in Q^k_<$,
  \begin{eqnarray*}
    Ae_{i_1} \wedge \ldots \wedge Ae_{i_k} &=& \left(\sum\limits_{h \in i_1^\downarrow} A_{h,i_1}e_h\right)\wedge \ldots \wedge \left(\sum\limits_{h \in i_k^\downarrow} A_{h,i_k}e_h\right) \\
     && \in \bigoplus\limits_{j \in I}\spn_{\F}\left\{e_{j_1} \wedge \ldots \wedge e_{j_k}\right\},
  \end{eqnarray*} where $I:=\left\{j\in [n]^k_<:j\preccurlyeq_{Q^k} i\right\}$.
\end{proof}

Let $\phi: \gr_{\F}(k,n) \rightarrow \mathbb{P}\bigl(\bigwedge^kV \bigr)$ be the Pl\"ucker embedding.
According to the action of the incidence group $I^*(Q^k_<;\F)$ on $\mathbb{P}\bigl(\bigwedge^kV \bigr)$, we provide an incidence stratification of the Grassmannian $\gr_{\F}(k,n)$.

\begin{dfn} \label{def schub grass}
  Let $Q_I(\F)$ be an orbit of the action of $I^*(Q^k_<;\F)$ on the projective space $\mathbb{P}\bigl(\bigwedge^kV \bigr)$, for any order ideal $I \in \mathcal{J}(Q^k_<)$. The set $$[Q]_I(\F):=(Q^k_<)_I(\F) \cap \phi(\gr_{\F}(k,n))$$ is called $Q$-\emph{Schubert cell} of $\gr_{\F}(k,n)$ whenever $[Q]_I(\F)\neq \varnothing$. A $Q$-\emph{Schubert variety} in $\gr_{\F}(k,n)$ is $\overline{[Q]_I(\F)}:=\overline{(Q^k_<)_I(\F)} \cap \phi(\gr_{\F}(n,k))$.
\end{dfn}

The next result follows directly from Definition \ref{def schub grass} and Theorem \ref{pspazio}.

\begin{prop}
 Let $I\in \mathcal{J}(Q^k_<)$ and $[Q]_I(\F)$ be a $Q$-Schubert cell.  We have that
 $$\overline{[Q]_I(\F)}=\biguplus\limits_{H \in \mathcal{J}(I)} [Q]_H(\F).$$
\end{prop}

Now we are going to prove that, in a Grassmannian variety, a $c_n$-Schubert cell is a Schubert cell. In other words, a Schubert cell
is the intersection of $\gr_{\F}(k,n)$ with a $(c_n)^k_<$-Schubert cell of the projective space $\mathbb{P}\bigl(\bigwedge^kV \bigr)$.

\begin{prop} \label{celle classiche GR}
 Let $\sigma \in S_n^{(k)}$ be a Grassmannian permutation and $C_\sigma(\F)$ the corresponding Schubert cell of $\gr_{\F}(k,n)$.
 Then $$C_\sigma(\F)= [c_n]_{I_\sigma}(\F),$$
 where $I_\sigma:=\left\{a\in [n]^k_<: a\preccurlyeq_{c^k_n}(\sigma(1),\ldots,\sigma(k))\right\}$.
\end{prop}
\begin{proof}
  A Schubert cell $C_\sigma(\F)$ in $\phi(\gr_{\F}(k,n))$ is an orbit under the action of $\pi^k_{c_n}(I^*(c_n;\F))$ of the line $\spn_{\F}\left\{e_{\sigma(1)}\wedge \ldots \wedge e_{\sigma(k)}\right\}$. By Theorem \ref{teorema sottogruppo Qk}, the group $\pi^k_{c_n}(I^*(c_n;\F))$ is a subgroup of $I^*((c_n)^k_<;\F)$; therefore the orbits of the action of $I^*((c_n)^k_<;\F)$ on $\mathbb{P}\bigl(\bigwedge^kV \bigr)$ are partitioned into orbits of $\pi^k_{c_n}(I^*(c_n;\F))$. But the Schubert cells give a partition of $\gr_{\F}(k,n)$, so the result follows.
\end{proof}

By the fact that a Schubert variety is union of Schubert cells according to the Bruhat order of $S^{(k)}_n$, the $c_n$-Schubert varieties in $\gr_{\C}(k,n)$ are exactly the Schubert varieties.

We define the set of $Q$-Schubert cells of $\gr_{\F}(k,n)$ as $$Q^k_B(\F):=\left\{I\in \mathcal{J}(Q^k_<):[Q]_I(\F)\neq \varnothing\right\}.$$
\begin{dfn} \label{Q Bruhat}
  Let $Q\in \pos(n)$. We call $(Q^k_B(\F),\subseteq)$ the $Q$-\emph{Bruhat poset} of $\gr_{\F}(k,n)$.
\end{dfn}

\begin{oss}
  By Proposition \ref{celle classiche GR}, the $c_n$-Bruhat poset of $\gr_{\F}(k,n)$ is isomorphic to $S_n^{(k)}$ with the Bruhat order,
  which by Proposition \ref{proposizione Bruhat quoziente} is isomorphic to $([n]^k_<,\leqslant_{c_n^k})$.
  Moreover, by Proposition \ref{refinem}, these posets are isomorphic to $(c_n)^k_<$.
\end{oss}

Notice that the poset $(Q^1_B(\F),\subseteq)$ is equal to $(\mathcal{J}(Q)\setminus \left\{\varnothing\right\},\subseteq)$, for all $Q\in \pos(n)$;
see Theorem \ref{pspazio}.

We provide a characterization of the $Q$-Schubert cells in terms of matroids representable over $\F$.

\begin{thm} \label{teorema matroidi}
  Let $Q\in \pos(n)$,  $k\in [n]$ and $I\in \mathcal{J}(Q^k_<)$. Then $[Q]_I(\F)\neq \varnothing$ if and only if
  $\max(I) \cup I'$ is a matroid representable over $\F$, for some subset $I' \subseteq I$.
\end{thm}
\begin{proof}
 Let $I'\subseteq I$ be any subset. The result follows by observing that $\max(I) \cup I'$ is the set of bases of a matroid representable over $\F$ if and only if there exists $A\in I^*(Q^k_<;\F)$ such that
 $$A\left(\sum \limits_{i\in \max(I)}e_{i_1}\wedge \ldots \wedge e_{i_k}\right)=\sum \limits_{i\in \max(I) \cup I'}a_{i}e_{i_1}\wedge \ldots \wedge e_{i_k}$$  is an element of  $\phi(\gr_{\F}(k,n))$, where $a_i \in \F\setminus \left\{0\right\}$ for all $i\in \max(I)$.
\end{proof}

\begin{oss} \label{oss weak order matroidi}
  By Theorem \ref{teorema matroidi}, $((t_n)_B^k(\F),\subseteq)$ is the poset of representable matroids on $\F$ of rank $k$ on the set $[n]$,
ordered by inclusion of the sets of bases. This is the so called \emph{weak order} on matroids, see e.g. \cite[Chapter~7]{oxley} and \cite[Chapter~9]{Theory of matroids}.
\end{oss}

\begin{oss}
  It follows by basic topology that the Zariski closure of the orbit corresponding to a matroid $M$ in the matroid stratification of $\gr_{\mathbb{C}}(k,n)$  is included in $\overline{[t_n]_M(\mathbb{C})}$.
  This inclusion can be strict as in \cite[Counterexample~2.6]{Ford}.
Notice that the defining ideal of the $t_n$-Schubert variety $\overline{[t_n]_M(\C)}$ in $\gr_{\F}(k,n)$ is the Grassmannian ideal $P_M$ of the matroid $M$, as defined in \cite[Section~3]{slack}.
\end{oss}

From the fact that a singleton $\left\{(i_1,\ldots,i_k)\right\}$ is always the set of bases of a matroid representable over any field, we deduce the following corollary.
\begin{cor} \label{corollario ideale principale}
  If $I\in \mathcal{J}(Q^k_<)$ is a principal order ideal, then $[Q]_I(\F)\neq \varnothing$.
\end{cor} The poset $(Q^k_B(\F),\subseteq)$ has maximum $[n]^k_<$, the uniform matroid; it is not difficult to see that its minimal elements, which correspond to the minima
of the poset $Q^k_<$, are the Grassmannian permutations $\sigma$ such that the $Q$-inversion number $\inv_Q(\sigma):=\left\{(i,j)\in [n]_<^2:\sigma(j)<_Q\sigma(i)\right\}$ is zero.

\begin{ex} \label{esempio rombo}
  Let $Q \in \pos(4)$ be the poset on $[4]$ such that $1 \vartriangleleft 2$, $1 \vartriangleleft 3$, $2\vartriangleleft 4$ and $3\vartriangleleft 4$. Then $Q^2_<$ is the following poset:
  
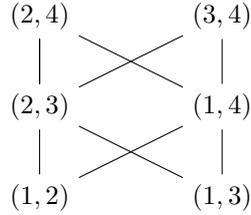
\begin{center}
\begin{tikzpicture}[scale=1.2]
  \node (a1) at (0,2) {$(2,4)$};
  \node (a2) at (2,2) {$(3,4)$};
  \node (b1) at (0,1) {$(2,3)$};
  \node (b2) at (2,1) {$(1,4)$};
 \node (c1) at (0,0) {$(1,2)$};
  \node (c2) at (2,0) {$(1,3)$};
  \draw (a1) -- (b1) -- (a1) -- (b2) (a2) -- (b1) (a2)--(b2) (b1)--(c1)(b1)--(c2) (b2)--(c1) (b2)--(c2);
\end{tikzpicture}
  \end{center}
 Let $(S_4,\left\{s_1,s_2,s_3\right\})$ be the symmetric group of order $24$ with its standard Coxeter presentation and $J:=\left\{s_1,s_3\right\}$. The $Q$-Bruhat on $\gr_{\C}(2,4)$ is then:
 
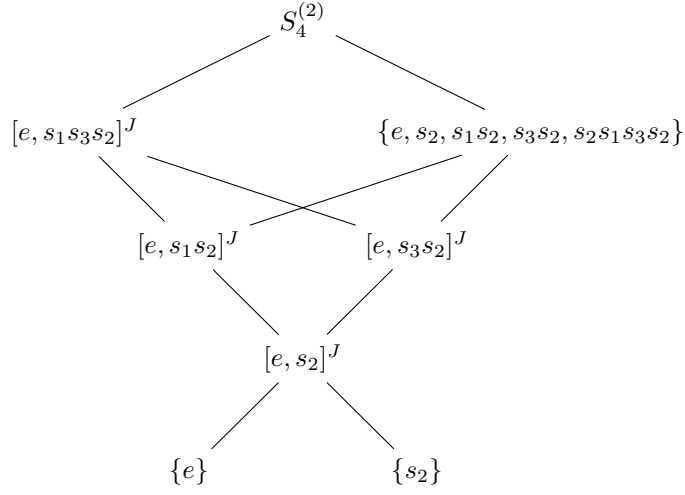
\begin{center}
\begin{tikzpicture}[scale=1.5]
 \node (a1) at (1,4) {$S_4^{(2)}$};
 \node (b1) at (0,3) {$[e,s_1s_3s_2]^J$};  \node (b2) at (2,3) {$s_2[e,s_1s_3s_2]^J$};
 \node (c1) at (0,2) {$[e,s_1s_2]^J$};
  \node (c2) at (2,2) {$[e,s_3s_2]^J$};
  \node (d1) at (1,1) {$\left\{e,s_2\right\}$};
  \node (e1) at (0,0) {$\left\{e\right\}$};
  \node (e2) at (2,0) {$\left\{s_2\right\}$};
  \draw (a1) -- (b1) (a1)--(b2) (b1) -- (c1) (b1)--(c2)(b2) -- (c1) (b2)--(c2) (c1)--(d1) (c2)--(d1) (d1)--(e1) (d1)--(e2);
\end{tikzpicture}
  \end{center}
where, if $u,v\in S_n^J$ and $u \leqslant v$, then $[u,v]^J:=\left\{z\in S^J_n: u \leqslant z \leqslant v\right\}$ is a Bruhat interval in the poset $(S^J_n,\leqslant)$
and $w[u,v]^J:=\left\{P^J(wz):z\in [u,v]^J\right\}$, for all $w\in S_n$.

By using the identification of $[4]^2_<$ with $S^{(2)}_4$, we have

\begin{itemize}
  \item $\left\{e\right\}=(1,2)^\downarrow$, $\left\{s_2\right\}=(1,3)^\downarrow$ and $\left\{e,s_2\right\}=\left\{(1,2),(1,3)\right\}^\downarrow$;
  \item $[e,s_1s_2]^J=(2,3)^\downarrow$ and $[e,s_3s_2]^J = (1,4)^\downarrow$;
  \item $[e,s_1s_3s_2]^J=(2,4)^\downarrow$ and $s_2[e,s_1s_3s_2]^J=(3,4)^\downarrow$.
\end{itemize}

Notice that, by Theorem \ref{teorema matroidi}, the order ideal $\left\{(2,3),(1,4)\right\}^\downarrow$
is not an element of the $Q$-Bruhat poset, i.e. $[Q]_I(\F)=\varnothing$.
\end{ex}

By Corollary \ref{corollario ideale principale}, $|Q_B^k(\F)|\geqslant |S_n^{(k)}|$. In the next proposition we obtain directly that the Bruhat order on $S_n^{(k)}$ is the $c_n$-Bruhat poset,
  without using Proposition \ref{celle classiche GR}.
\begin{prop} \label{corollario iniezione}
 Let $n>0$ and $k\in [n]$. Then $(S_n^{(k)},\leqslant)\simeq ((c_n)_B^k(\F),\subseteq)$.
\end{prop}
\begin{proof}


 Let $I_1,I_2\in \mathcal{J}((c_n)^k_<)$ such $I_1 \subseteq I_2$ and $\max(I_1)\cup I'_1$, $\max(I_2)\cup I'_2$ are matroids representable over $\F$ for some subsets $I_1'\subseteq I_1$, $I_2'\subseteq I_2$. Since the Gale order $\leqslant^e$ on $[n]_<^k$ is $\preccurlyeq_{c_n^k}$, by the Maximality Property of matroids, $|\max[\max(I_1)\cup I_1']|=|\max(I_1)|=1$ and $|\max[\max(I_2)\cup I_2']|=|\max(I_2)|=1$.

 Moreover, if $\max(I) \in (c_n)_<^k \simeq S_n^{(k)}$ and $|\max(I)|=1$, then $\max(I)$ is clearly a representable matroid. Hence, by Theorem \ref{teorema matroidi}, $I \in (c_n)_B^k(\F)$.
\end{proof}

It is natural to go on with further investigations on the $Q$-Bruhat orders introduced in this section. For instance, supported by several computational examples, we formulate a conjecture.

\begin{conge} \label{congettura 1}
  Let $Q\in \pos(n)$ and $k\in [n]$. Then the poset $(Q^k_B(\mathbb{C}),\subseteq)$ is graded with rank function $\rho(I)=\dim([Q]_I(\mathbb{C}))$,
  for all $I\in Q^k_B(\mathbb{C})$.
\end{conge}

Conjecture \ref{congettura 1} holds when $Q=c_n$, since the Bruhat order on the quotients is graded with rank function the inversion number of the permutation.
For $k=1$ the conjecture holds for every poset $Q$, by Corollary \ref{corollario dimensione pspazio}.
The dimension of the $t_n$-Schubert cells in $\gr_{\C}(k,n)$ is provided by \cite[Theorem~2.5]{transcendence degree}.

\section{Incidence stratifications of P-flag spaces} \label{sezione strati pflag}

In this section we study incidence stratifications of $\pfl(\F)$, for every field $\F$. In order to do this we embed $\pfl(\F)$ in a projective space
and we need to construct suitable posets.

Recall that $V=\spn_{\mathbb{F}}\left\{e_i:i\in [n]\right\}$.
Let $P\in \pos(n)$ and consider the function $$\phi_P: \pfl(\F)\rightarrow \mathbb{P}\left(\bigotimes\limits_{i=1}^n \bigwedge^{|i^\downarrow_P|}V\right), $$
induced by the assignment $$F \mapsto \left(v^1_1 \wedge \ldots \wedge v^1_{|1^\downarrow|} \right)\otimes \ldots \otimes \left(v^n_1 \wedge \ldots \wedge v^n_{|n^\downarrow|}\right),$$
for all $F \in \pfl(\F)$, where $\left\{v^i_1,\ldots,v^i_{|i^\downarrow|}\right\}$ is any basis of $F_i$, for all $i\in [n]$.
It is easy to see that this function is injective.

Let $Q\in \pos(n)$. There exists a
representation $$\pi_Q:I^*(Q;\F)\rightarrow \aut\left(\bigotimes\limits_{i=1}^n \bigwedge^{|i^\downarrow_P|}V \right)$$ obtained extending the action of $I^*(Q;\F)$
on $V$:
$$A\left(v^1_1 \otimes (v^2_1 \wedge \ldots \wedge v^2_{|2^\downarrow|}) \otimes \ldots \otimes (v^n_1 \wedge \ldots \wedge v^n_{|n^\downarrow|})\right)$$$$=(Av^1_1) \otimes (Av^2_1 \wedge \ldots \wedge Av^2_{|2^\downarrow|}) \otimes \ldots \otimes (Av^n_1 \wedge \ldots \wedge Av^n_{|n^\downarrow|}),$$ for all $A\in I^*(Q;\F)$.

\begin{dfn} \label{def Q alla P}
  Let $Q,P \in \pos(n)$. We define a poset $$Q^P:=Q \times Q_<^{|2^\downarrow_P|} \times \ldots \times Q_<^{|n^\downarrow_P|}.$$
\end{dfn}
By Proposition \ref{proposizione lex}, it is clear that the lexicographic order on $Q^P$ provide a natural labeling
and then we consider $Q^P\in \pos\left(|Q^P|\right)$.

\begin{thm} \label{teorema sottogruppo QP}
  The group morphism $\pi_Q$ is injective and $\pi_Q(I^*(Q;\F))$ is a subgroup of the incidence group $I^*(Q^P;\F)$.
\end{thm}
\begin{proof}
  Let $A\in I^*(Q;\F)$ be such that $\pi_Q(A)=\id$. Then $Av_1 \in \spn_{\F}\left\{v_1\right\}$ for all $v_1 \in V$, i.e.\ $A$ is the identity matrix.
  The other assertion follows by Theorem \ref{teorema sottogruppo Qk}.
\end{proof}
We can decompose the projective space $\mathbb{P}\left(\bigotimes\limits_{i=1}^n \bigwedge^{|i^\downarrow_P|}V\right)$ according to
the action of the incidence group $I^*(Q^P;\F)$, giving an incidence stratification of $\pfl(\F)$.

\begin{dfn} \label{def celle bandiere} Let $Q^P_I(\F)$ be an orbit of the action of $I^*(Q^P;\F)$ on the projective space $\mathbb{P}\left(\bigotimes\limits_{i=1}^n \bigwedge^{|i^\downarrow|}V\right)$, for any order ideal
  $I \in \mathcal{J}(Q^P)$. The set
   $$[Q]^P_I(\F):=Q^P_I(\F) \cap \phi_P(\pfl(\F))$$ is called $Q$-\emph{Schubert cell} of $\pfl(\F)$, whenever $[Q]^P_I(\F)\neq \varnothing$. A  $Q$-\emph{Schubert variety} in $\pfl(\F)$ is defined by $\overline{[Q]^P_I(\F)}:=\overline{Q^P_I(\F)} \cap \phi_P(\pfl(\F))$.
\end{dfn}
The next result follows directly from Definition \ref{def celle bandiere} and Theorem \ref{pspazio}.
\begin{prop} \label{prop chiusure}
 Let $I\in \mathcal{J}(Q^P)$ and $[Q]^P_I(\F)$ be a $Q$-Schubert cell of $\pfl(\F)$.  We have that
 $$\overline{[Q]^P_I(\F)}=\biguplus\limits_{H \in \mathcal{J}(I)} [Q]^P_H(\F).$$
\end{prop}

The following proposition asserts that, in a flag variety,
a $c_n$-Schubert cell is a Schubert cell. In other words, a Schubert cell is the
intersection of $\mathrm{Fl}_n(\F)$ with a $(c_n)^{c_n}$-cell of the projective space $\mathbb{P}\left(\bigotimes\limits_{i=1}^n \bigwedge^{i}V\right)$.
\begin{prop} \label{celle classiche F} Let $\sigma \in S_n$ and
 $C_\sigma(\F)$ be the corresponding Schubert cell of $\mathrm{Fl}_n(\F)$. Then $$C_\sigma(\F)= [c_n]^{c_n}_{I_\sigma}(\F),$$
 where the principal order ideal $I_\sigma$ of $(c_n)^{c_n}$ is defined by $$I_\sigma :=(\left\{\sigma(1)\right\}_<,\left\{\sigma(1),\sigma(2)\right\}_<,\ldots,\left\{\sigma(1),\sigma(2),\ldots,\sigma(n)\right\}_<)^\downarrow$$
 and $\left\{x_1,\ldots,x_h\right\}_< \in [n]^h_<$ is the tuple obtained ordering $x_1,\ldots,x_h$.
\end{prop}
\begin{proof}
  A Schubert cell $C_\sigma(\F)$ in $\phi_{c_n}(\mathrm{Fl}_n(\F))$ is an orbit of the flag $\phi_{c_n}(F_\sigma)$  under the action of $\pi_Q(I^*(c_n;\F))$. By Theorem \ref{teorema sottogruppo QP}, the group $\pi_Q(I^*(c_n;\F))$ is a subgroup of $I^*((c_n)^{c_n};\F)$ and we conclude as in the proof of Proposition \ref{celle classiche GR}.
\end{proof}
We define the set of $Q$-Schubert cells of $\pfl(\F)$ as $$Q^P_B(\F):=\left\{I\in \mathcal{J}(Q^P):[Q]^P_I(\F)\neq \varnothing\right\}.$$
\begin{dfn} \label{QP Bruhat order}
  Let $P,Q\in \pos(n)$. We call $(Q^P_B(\F),\subseteq)$ the $Q$-\emph{Bruhat poset} of $\pfl(\F)$.
\end{dfn}
By Propositions \ref{prop chiusure} and \ref{celle classiche F}, the $c_n$-Bruhat poset of $\mathrm{Fl}_n(\F)$ is isomorphic to $S_n$ with the Bruhat order.

In order to characterize the $Q$-Schubert cells of $\pfl(\F)$ we need to introduce a Gale order on the underlying set of $Q^P$.

\subsection{$P$-flags and the Maximality Property}

Let $P\in \pos(n)$ and define the set $$[n]^P:=\prod\limits_{i=1}^n[n]_<^{|i^\downarrow|}.$$ For example, we have that $[n]^{t_n}\simeq [n]^n$.
The symmetric group $S_n$ acts on $[n]^P$ by setting
$$\sigma\left(\left(i_{1,1}\right),\left(i_{2,1},\ldots,i_{2,|2^\downarrow|}\right),\ldots,\left(i_{n,1},\ldots,i_{n,|n^\downarrow|}\right)\right):=$$
$$\left(\left\{\sigma(i_{1,1})\right\}_<,\left\{\sigma(i_{2,1}),\ldots,\sigma(i_{2,|2^\downarrow|})\right\}_<,\ldots,\left\{\sigma(i_{n,1}),\ldots,\sigma(i_{n,|n^\downarrow|})\right\}_<\right),$$
for all $\sigma \in S_n$.

We introduce the following useful order on $[n]^P$.

\begin{dfn}
  The \emph{Gale ordering} $\leqslant^\sigma_P$  on $[n]^P$ is defined by letting
  $$a \leqslant^{\sigma}_P b \Longleftrightarrow \sigma a \leqslant_{(c_n)^P} \sigma b,$$ for all $a,b\in [n]^P$.
\end{dfn}
In particular, $([n]^P,\leqslant^e_P)=(c_n)^P$ and the Gale ordering  $\leqslant^e_{t_n}$ on $[n]^{t_n}$ is $[n]^n$ ordered componentwise.

\begin{oss} \label{oss Q^P Gale}
  As in the proof of Proposition \ref{proposizione lex},
  we have that $Q^P \hookrightarrow (c_n)^P$, for all $P\in \pos(n)$. In particular, $Q^P \hookrightarrow ([n]^P,\leqslant^e_P)$.
\end{oss}

The following definitions are crucial for the study of incidence stratifications of a $P$-flag space, see Section \ref{sezione Q-strat P-flag}.

\begin{dfn} \label{def max prop bandiere}
  A subset $\mathcal{F}\subseteq [n]^P$ has the \emph{Maximality Property} if the poset $\left(\mathcal{F},\leqslant^\sigma_P\right)$ has maximum
   for all $\sigma\in S_n$.
\end{dfn}

\begin{oss}
  By Remark \ref{osserv massimi minimi}, if $\mathcal{F}\subseteq [n]^P$ has the Maximality Property, then the poset $\left(\mathcal{F},\leqslant^\sigma_P\right)$ has minimum
   for all $\sigma\in S_n$.
\end{oss}

\begin{oss}
  It should be clear that, by definition, given $\mathcal{F}\subseteq [n]^P$ with the Maximality Property, the set $M_i:=\left\{F_i:F \in \mathcal{F}\right\}$ is a matroid of rank $|i^\downarrow|$, for all $i\in [n]$.
\end{oss}


Recall from Section \ref{subsection matroids} that $M(W)$ stands for the matroid represented by the vector space $W$.
\begin{dfn} \label{def flag repr}
  We say that $\mathcal{G} \subseteq [n]^P$ is represented by a $P$-flag $F\in \pfl(\F)$ if
  $\mathcal{G}=M(F_1)\times \ldots \times M(F_n)$.
\end{dfn}
As usual we identify $\biguplus \limits_{k=0}^n[n]^k_<$ with the power set $\mathcal{P}([n])$. We define the following subset of $[n]^P$:
$$[n]^P_{\subseteq}:=\left\{a\in [n]^P:i<_Pj \Rightarrow a_i \subsetneq a_j,\,\forall \,i,j\in [n]\right\}.$$

\begin{thm} \label{teorema P-matroidi rappresentabili}
  Let $\mathcal{G} \subseteq [n]^P$ be represented by a $P$-flag $F\in \pfl(\F)$.
  Then $\mathcal{G}$ has the Maximality Property and
  its $\leqslant_P^{\sigma}$-maximum lies in $[n]^P_{\subseteq}$,  for all $\sigma \in S_n$.
\end{thm}
\begin{proof}
By Definition \ref{def flag repr}, $\mathcal{G}_h=M(F_h)$ is a matroid, for every $h\in [n]$.
Let $\sigma \in S_n$ and $m_h^\sigma \in \mathcal{G}_h$ be the maximum of the poset $(\mathcal{G}_h,\leqslant^\sigma)$, for all $h\in [n]$. It is clear that $(m_1^{\sigma},\ldots,m_n^{\sigma})$ is
the maximum of $(\mathcal{G},\leqslant^\sigma_P)$.

Let $i<_P j$. Then $|i^\downarrow|<|j^\downarrow|$ and $(F_i,F_j)$ is a partial flag; by \cite[Theorem~1.7.3]{coxeter matroids}, the matroids $M(F_i)$ and  $M(F_j)$ are concordant (see \cite[Section~1.7.3]{coxeter matroids}). By \cite[Corollary~1.7.2]{coxeter matroids}, the pair $\left(m_i^\sigma,m_j^\sigma\right)$ satisfies $m_i^\sigma \subsetneq m_j^\sigma$, for all $\sigma \in S_n$.
\end{proof}

%
%
The next result provides a fundamental tool to describe the $t_n$-stratification of a $P$-flag space, see Corollary \ref{cor P-flag matroidi}.
\begin{prop} \label{prop intersezione}
  Let $\mathcal{F},\mathcal{G} \subseteq [n]^P$ represented by $P$-flags.
  Then $$\mathcal{F} \cap [n]^P_\subseteq = \mathcal{G} \cap [n]^P_\subseteq \Longrightarrow \mathcal{F}=\mathcal{G}.$$
\end{prop}
\begin{proof}
  By contradiction, assume $ \mathcal{F}\neq \mathcal{G}$ and let $m \in \mathcal{F}_i \setminus \mathcal{G}_i$, for some $i\in [n]$.
  Then there exists $\sigma \in S_n$ such that $m=\max\left(\mathcal{F}_i,\leqslant^\sigma\right)$.
  By Theorem \ref{teorema P-matroidi rappresentabili}  there exists $a:=\max\left(\mathcal{F},\leqslant^\sigma_P\right) \in [n]^P_\subseteq$ and,
  by hypothesis, $a\in \mathcal{G}$. Hence $m=a_i \in  \mathcal{G}_i$, a contradiction.
\end{proof}


\subsection{$Q$-stratification of $P$-flag spaces} \label{sezione Q-strat P-flag}

Now we are able to provide a characterization of $Q$-Schubert cells in the space $\pfl(\F)$.

\begin{thm} \label{teorema P-matroidi bandiera}
  Let  $P,Q\in \pos(n)$ and $I\in \mathcal{J}(Q^P)$. Then $[Q]^P_I(\F)\neq \varnothing$ if and only if
 there exists $I' \subseteq I$ such that $\max(I) \cup I'$ is represented by some $F\in \pfl(\F)$.
\end{thm}
\begin{proof}
Let $I'\subseteq I$ be any subset. The result follows by observing that $\max(I) \cup I'=M(F_1)\times \ldots \times M(F_n)$  for some $F\in \pfl(\F)$ if and only if there exists $A\in I^*(Q^P;\F)$ such that
 $$A\left(\sum \limits_{i\in \max(I)}e_{i_1} \otimes \ldots \otimes e_{i_n}\right)=\sum \limits_{i\in \max(I) \cup I'}a_{i}e_{i_1} \otimes \ldots \otimes e_{i_n}=\phi_P(F),$$ where $a_i \in \F\setminus \left\{0\right\}$ for all $i\in \max(I)$
 and we have defined $$e_x:=e_{x_1}\wedge \ldots \wedge e_{x_k}$$ for all $x\in [n]^k_<$.
\end{proof}

%
%
%

%
%

As a consequence of Theorem \ref{teorema P-matroidi bandiera} we recover the Bruhat order of $S_n$.
\begin{prop} \label{corollario iniezione2}
 We have that $\left(S_n,\leqslant\right) \simeq \left([n]^{c_n}_\subseteq,\leqslant_{(c_n)^{c_n}}\right) \simeq \left((c_n)^{c_n}_B(\F),\subseteq\right)$.
\end{prop}
\begin{proof}
The first poset isomorphism is clear by definition.
Let $I_1,I_2\in \mathcal{J}((c_n)^{c_n})$ such that $I_1 \subseteq I_2$ and $\max(I_1)\cup I'_1$, $\max(I_2)\cup I'_2$ are represented by flags, for some subsets $I_1'\subseteq I_1$, $I_2'\subseteq I_2$. Since the $c_n$-Gale order $\leqslant^e_{c_n}$ on $[n]^{c_n}$ is $(c_n)^{c_n}$, by the Maximality Property, $|\max[\max(I_1)\cup I_1']|=|\max(I_1)|=1$ and $|\max[\max(I_2)\cup I_2']|=|\max(I_2)|=1$. By Theorem \ref{teorema P-matroidi rappresentabili} we have that $\max(I_1),\max(I_2)\in [n]^{c_n}_\subseteq$.
Moreover, if $\max(I) \in [n]_\subseteq^{c_n}$ then $\max(I)$ is clearly represented by a flag. Hence, by Theorem \ref{teorema P-matroidi bandiera}, $I \in (c_n)^{c_n}_B(\F)$.
\end{proof}

The following example shows the stratification of $\mathrm{Fl}_3(\F)$ induced by the action on the projective space $\mathbb{P}\left[V\otimes \left(V\wedge V\right)\right]$ of the group $I^*(Q \times Q^2_<;\F)$, where $Q\in \pos(3)$ is one of the posets of Example \ref{esempio poset V}. In this case the factor $V\wedge V \wedge V$ is redundant.

\begin{ex}
  Let $P=c_3$ and $Q\in \pos(3)$ be the poset whose cover relations are $1 \vartriangleleft_Q 3$ and $2 \vartriangleleft_Q 3$.
  Then the poset $Q\times Q^2_<$ has the following Hasse {diagram}\footnote{We omit parentheses when writing the elements of $[n]_<$ and $[n]^2_<$.}:
  \begin{center}
 \begin{tikzpicture}[scale=1.2]
  \node (a1) at (2,2) {$(3,13)$};
  \node (a2) at (6,2) {$(3,23)$};
  \node (b1) at (0,1) {$(1,13)$};
  \node (b2) at (2,1) {$(1,23)$};
  \node (b3) at (4,1) {$(3,12)$};
  \node (b4) at (6,1) {$(2,13)$};
  \node (b5) at (8,1) {$(2,23)$};
  \node (c1) at (2,0) {$(1,12)$};
  \node (c2) at (6,0) {$(2,12)$};
  \draw (a1) -- (b1) -- (a1) -- (b3) -- (a1) -- (b4) (a2)--(b2)--(a2)--(b3)--(a2)--(b5) (b1)--(c1) (b2)--(c1) (b3)--(c1) (b3)--(c2) (b4)--(c2) (b5)--(c2);
\end{tikzpicture}
  \end{center}

By Remark \ref{oss Q^P Gale} and Theorem \ref{teorema P-matroidi rappresentabili}, the principal order ideals of $Q\times Q^2_<$ which satisfy the condition of Theorem \ref{teorema P-matroidi bandiera}
are the ones with maximum in the set $$\left\{(1,12),(2,12),(1,13),(2,23),(3,13),(3,23)\right\},$$ which corresponds to the symmetric group $S_3$. We consider $S_3$ with its standard Coxeter presentation with generators $\left\{s,t\right\}$. Then $s=213=(2,12)$, $t=132=(1,13)$, $st=231=(2,23)$, $ts=312=(3,13)$ and $sts=321=(3,23)$.

Using Theorem \ref{teorema P-matroidi rappresentabili}, the
non-principal order ideals to be considered are $\left\{(1,12),(2,12)\right\}$, which is represented by the flag $$\left(\spn_{\F}\left\{e_1+e_2\right\}, \spn_{\F}\left\{e_1\wedge e_2\right\}\right),$$ and
the order ideal $I=\left\{(1,13),(1,23),(2,13),(2,23)\right\}^\downarrow$. We have that $\max(I)$ is represented by the flag
   $$\left(\spn_{\F}\left\{e_1+e_2\right\},\spn_{\F}\left\{(e_1+e_2)\wedge e_3\right\}\right).$$
Therefore the poset $(Q^{c_n}_B,\subseteq)$ has the following Hasse diagram:
 \begin{center}
 \begin{tikzpicture}[scale=1]
  \node (a1) at (2,3) {$[n]^{c_n}$};
  \node (b1) at (0,2) {$ts^\downarrow $};
  \node (b2) at (2,2) {$I$};
  \node (b3) at (4,2) {$sts^\downarrow $};
 \node (c1) at (0,1) {$\left\{e,t\right\}$};
  \node (c2) at (2,1) {$\left\{e,s\right\}$};
  \node (c3) at (4,1) {$\left\{s,st\right\}$};
  \node (d1) at (1,0) {$\left\{e\right\}$};
  \node (d2) at (3,0) {$\left\{s\right\}$};
  \draw (a1) -- (b1) (a1) -- (b2) (a1)--(b3) (b1)--(c1) (b1)--(c2) (b2)--(c1) (b2)--(c2) (b2)--(c3) (b3)--(c2) (b3)--(c3) (c1)--(d1) (c2)--(d1)  (c2)--(d2) (c3)--(d2);
\end{tikzpicture}
  \end{center}
\end{ex}

Notice that the action of $S_n$ on $[n]^P$ restricts to an action on $[n]^P_{\subseteq}$. We are ready to introduce the notion of $P$-flag matroid which extends the one of flag matroid, see Definition \ref{def flag matr}.
\begin{dfn} \label{definizione P flag matroid}
  A subset $\mathcal{F}\subseteq [n]_\subseteq^P$ is a $P$-\emph{flag matroid} if it has the Maximality Property.
\end{dfn}
\noindent The set $[n]_\subseteq^P$ is a $P$-flag matroid, which we call \emph{uniform $P$-flag matroid}.
Notice that $c_n$-flag matroids coincide with flag matroids in $S_n$.

\begin{ex} \label{esempio tn matroidi}
 The uniform $t_2$-matroid is $[2]^2=\left\{(1,1),(1,2),(2,1),(2,2)\right\}$. We list all the $t_2$-matroids $\mathcal{F}\subsetneq [2]^2$:
  \begin{enumerate}
    \item $|\mathcal{F}|=1$: $\left\{(1,1)\right\}$,  $\left\{(1,2)\right\}$, $\left\{(2,1)\right\}$, $\left\{(2,2)\right\}$.
    \item $|\mathcal{F}|=2$: $\left\{(1,1),(1,2)\right\}$, $\left\{(1,1),(2,1)\right\}$, $\left\{(1,1),(2,2)\right\}$, $\left\{(1,2),(2,2)\right\}$,

    $\left\{(2,1),(2,2)\right\}$.
    \item $|\mathcal{F}|=3$: $\left\{(1,1),(1,2),(2,2)\right\}$, $\left\{(1,1),(2,1),(2,2)\right\}$.
  \end{enumerate}
 For instance, the set $\left\{(1,2),(2,1)\right\}$ is not a $t_2$-matroid.
\end{ex}

\begin{dfn} \label{def P-flag matroid rep}
  A $P$-flag matroid $\mathcal{F}$ is called \emph{representable} over $\F$, if there exists $\mathcal{G} \subseteq [n]^P$
  represented by $F\in \pfl(\F)$ such that $\mathcal{F}=\mathcal{G}\cap [n]^P_\subseteq$.
\end{dfn}

\begin{ex}
  Since $[n]^{t_n}_\subseteq =[n]^{t_n}$, then $((t_n)^{t_n}_B(\F),\subseteq)$ is the poset of $\F$-representable $t_n$-flag matroids. The $\F$-representable $t_2$-flag matroids are $\left\{(1,2)\right\}$, $\left\{(2,1)\right\}$, $\left\{(1,1),(1,2)\right\}$, $\left\{(1,1),(2,1)\right\}$, $\left\{(1,2),(2,2)\right\}$, $\left\{(2,1),(2,2)\right\}$ and
  the uniform one. The Hasse diagram of $(t_2)^{t_2}_B$
  is
 \begin{center}
 \begin{tikzpicture}[scale=1]
  \node (a1) at (2,2) {$\bullet$};
  \node (b1) at (0,1) {$\bullet$};
  \node (b2) at (1,1) {$\bullet$};
  \node (b3) at (3,1) {$\bullet$};
  \node (b4) at (4,1) {$\bullet$};
  \node (c1) at (1,0) {$\bullet$};
  \node (c2) at (3,0) {$\bullet$};
  \draw (a1) -- (b1) (a1) -- (b2) (a1) -- (b3)(a1) -- (b4)(b1)--(c1) (b2)--(c1) (b3)--(c2) (b4)--(c2);
\end{tikzpicture}
  \end{center}
\end{ex}

The following results extend the flag matroid stratification of a flag variety.

\begin{prop} \label{prop P-flag matroids}
  Let  $P\in \pos(n)$ and $I,J\in \mathcal{J}((t_n)^P)$ such that $[t_n]^P_I(\F)\neq \varnothing$ and $[t_n]^P_J(\F)\neq \varnothing$. Then
  $$I \cap [n]^P_\subseteq =J \cap [n]^P_\subseteq \Longrightarrow [t_n]^P_I(\F)=[t_n]^P_J(\F).$$
\end{prop}
\begin{proof}
 By Theorem \ref{teorema P-matroidi bandiera},  we have that $I=\max(I)$ and $J=\max(J)$ are represented by some $P$-flags $F$ and $G$, respectively.
 The result follows by Proposition \ref{prop intersezione}.
\end{proof}

\begin{cor} \label{cor P-flag matroidi}
The set $(t_n)_B^P(\F)$ is in bijection with the set of $\F$-representable $P$-flag matroids.
\end{cor}
\begin{proof}
 Let $I \in \mathcal{J}\left((t_n)^P\right)$.  By Theorem \ref{teorema P-matroidi bandiera}, $I\in (t_n)_B^P(\F)$ if and only if it is represented by a $P$-flag.
  By Theorem \ref{teorema P-matroidi rappresentabili}, $I$ has the Maximality Property. Then $I \cap [n]^P_\subseteq$ is a $P$-flag matroid and, by Definition \ref{def P-flag matroid rep}, it is representable over $\F$. Hence the result follows by Proposition \ref{prop P-flag matroids}.
\end{proof}

We conclude with the following conjecture.
\begin{conge} \label{congettura 2}
  Let $n>0$ and $Q,P\in \pos(n)$. Then the poset $(Q^P_B(\mathbb{C}),\subseteq)$ is graded.
\end{conge}

By Proposition \ref{corollario iniezione2},  when $Q=P=c_n$ Conjecture \ref{congettura 2} holds, since the Bruhat order on $S_n$ is graded.
Also for $P=t_n$ and $Q=c_n$ the poset is graded, see Corollary \ref{corollario parking}.

\subsection{The $t_n$-flag space and its parking function stratification}

In this section we provide an incidence stratification of a $t_n$-flag space by \emph{parking functions}. We refer to \cite[Exercise~5.49]{StaEC2}, \cite{Van Willigenburg} and
\cite{Yan parking} for further details and references on parking functions.

\begin{dfn}
 A \emph{parking function} over $n$ is an element $a\in [n]^n$ such that $(a_1,\ldots,a_n) \leqslant_{c_n^n} (\sigma(1),\ldots, \sigma(n))$,
  for some permutation $\sigma \in S_n$.
\end{dfn}
For example $(4,1,1,1,2,6,4)$ is a parking function over $7$ whereas the element $(6,6,6,1,2,3,4)$ is not a parking function.
\begin{dfn}
 Let $a\in [n]^n$. If $a \geqslant_{c_n^n} (\sigma(1),\ldots, \sigma(n))$ for some $\sigma \in S_n$, we say that $a$ is a \emph{dual parking function} over $n$.
\end{dfn} For example $(6,3,5,1,2,7,7)$ is a dual parking function over $7$ whereas the element $(1,2,2,2,2,4,3)$ is not a dual parking function. Notice that the self-dual parking functions are the permutations.

In the following theorem we describe the $c_n$-stratification of the space $\mathrm{Fl}_{t_n}(\F)$.

\begin{thm} \label{proposizione parking}
 Let $I$ be an order ideal of $(c_n)^{t_n}$. Then $[c_n]^{t_n}_I(\F)\neq \varnothing$ if and only if $|\max(I)|=1$ and $\max(I)$ is a dual parking function.
\end{thm}
\begin{proof}
Notice that subsets of $[n]^{t_n}$ represented by $t_n$-flags coincide with representable $t_n$-flag matroids.
 By Remark \ref{oss Q^P Gale}, only principal order ideals have to be considered in Theorem \ref{teorema P-matroidi bandiera}, for the other ones have more than one maximal element.

  A representable $t_n$-flag matroid over $\F$ is represented by an element $v:=v_1 \otimes \ldots \otimes v_n \in \bigotimes\limits_{i=1}^n \F^n$,
  such that $v_1,\ldots,v_n$ are linearly independent or, equivalently, by a matrix $M(v) \in \gl(n,\F)$ with columns $v_1,\ldots,v_n$.
  Therefore $$v=\sum \limits_{i \in [n]^n}\left((v_1)_{i_1}\cdots (v_n)_{i_n} \right)e_{i_1} \otimes \ldots \otimes e_{i_n},$$
  where $(v_i)_j\in \F$ is the $j$-th component of the vector $v_i$, for all $i,j\in [n]$. It is clear that, since $M(v)\in  \gl(n,\F)$, there exists $\sigma \in S_n$ such that $(v_1)_{\sigma(1)}\cdots (v_n)_{\sigma(n)} \neq 0$.

  If $[c_n]^{t_n}_I(\F)\neq \varnothing$ then, by our previous considerations and Theorem \ref{teorema P-matroidi bandiera}, there exists $\sigma \in S_n$ such that
 $(\sigma(1),\ldots,\sigma(n)) \in I$. This implies $\max(I)\geqslant_{c_n^n} (\sigma(1),\ldots,\sigma(n))$; so $\max(I)$ is a dual parking function.

 On the other hand, if $\max(I)=a$ is a dual parking function then $a\geqslant_{c_n^n} (\sigma(1),\ldots,\sigma(n))$ for some $\sigma \in S_n$
 and the vector
 \begin{eqnarray*}
   v &:=& \sum\limits_{(\sigma(1),\ldots,\sigma(n)) \leqslant_{c_n^n} b \leqslant_{c_n^n} a} e_{b_1}\otimes \ldots \otimes e_{b_n} \\
    &=& \left(\sum \limits_{\sigma(1) \leqslant i_1 \leqslant a_1} e_{i_1}\right) \otimes \ldots \otimes \left(\sum \limits_{\sigma(n) \leqslant i_n \leqslant a_n} e_{i_n}\right)
 \end{eqnarray*}
 represents over $\F$ a $t_n$-flag matroid, since the matrix $M(v)$ is equivalent to an invertible upper triangular matrix. Then the condition of Theorem \ref{teorema P-matroidi bandiera} is satisfied.
\end{proof}

\begin{oss}
  The $t_n$-flag space has been stratified by permutations in \cite{A flag variety for the Delta Conjecture} by gluing the orbits of the left action of the group of lower triangular matrices (more in general the varieties $X_{n,k}$ studied there have been stratified by \emph{Fubini words}, which reduce to permutations when $k=n$).
\end{oss}

\begin{oss}
 An incidence stratification of $\mathrm{Fl}_{t_n}\left(\F\right)$ made of parking function can be obtained by considering the action of the group of lower triangular matrices.
\end{oss}

\begin{oss}
  By Theorem \ref{proposizione parking}, an analog of Corollary \ref{corollario ideale principale} for $Q$-Schubert cells of $\pfl(\F)$ does not hold.
\end{oss}

\begin{cor} \label{corollario parking}
  The poset $((c_n)^{t_n}_B(\F),\subseteq)$ has cardinality $(n+1)^{n-1}$ and it is graded, with rank function $\rho(a):=\sum \limits_{i=1}^n(a_i-i)$, for any dual parking function $a\in [n]^n$.
\end{cor}

\section{Acknowledgements}

The first author was partially supported by Swiss National Science Foundation Professorship grant PP00P2\_179110/1 of Prof. Emanuele Delucchi.

He is grateful to the town of Zagarolo, where this paper  started and finished, for the hospitality received there.

\end{document}